\documentclass{amsart}
\usepackage{mathrsfs}
\usepackage{fullpage}
\usepackage{caption}
\usepackage{latexsym,amsfonts,amssymb,epsfig,verbatim}
\usepackage{amsmath, latexsym, graphics, textcomp}
\usepackage{pinlabel}

 \usepackage{graphicx,enumerate}
 \usepackage{tikz}
  \usetikzlibrary{cd}

\newcommand{\tips}{{Stealth[length=3mm, width=2mm]}}
\newcommand{\tipss}{{Stealth[length=2.5mm, width=1.5mm]}}

\newcommand{\nc}{\newcommand}
\nc{\dmo}{\DeclareMathOperator}
\nc{\nt}{\newtheorem}

\newtheorem{theorem}{Theorem}[section]
\newtheorem{lemma}[theorem]{Lemma}
\newtheorem{corollary}[theorem]{Corollary}

\newtheorem{proposition}[theorem]{Proposition}

\newtheorem{claim}[theorem]{Claim}

\theoremstyle{definition}
\newtheorem{remark}[theorem]{Remark}
\newtheorem{definition}[theorem]{Definition}

\nc{\Mod}{\mbox{\rm{Mod}}}
\nc{\PMod}{\mbox{\rm{PMod}}}
\nc{\Homeo}{\mbox{\rm{Homeo}}}
\nc{\Diff}{\mbox{\rm{Diff}}}
\nc{\Aut}{\mbox{\rm{Aut}}}
\nc{\Out}{\mbox{\rm{Out}}}
\nc{\Hom}{\mbox{\rm{Hom}}}
\nc{\Ker}{\mbox{\rm{Ker}}}
\nc{\PSL}{\mbox{\rm{PSL}}}
\nc{\Inn}{\mbox{\rm{Inn}}}
\nc{\Curr}{\mbox{\rm{Curr}}}
\nc{\QI}{\mbox{\rm{QI}}}
\nc{\Isom}{\mbox{\rm{Isom}}}

\nc{\rk}{\mbox{\rm{rk}}}
\nc{\dis}{\mbox{\rm{d}}}
\nc{\bn}{B_n}
\nc{\pbn}{PB_n}
\nc{\cn}{A(B_{n-1})}
\nc{\hGa}{\hat{\Ga}}
\nc{\Q}{\mathbb{Q}}
\nc{\G}{\mathcal{G}}
\nc{\K}{\mathcal{K}}
\nc{\T}{\mathcal{T}}
\nc{\hG}{\hat{\G}}
\nc{\h}{\frak{H}}
\nc{\PML}{\mathcal{PML}}
\nc{\ML}{\mathcal{ML}}
\nc{\EL}{\mathcal{EL}}
\nc{\PMF}{\mathcal{PMF}}
\dmo{\fh}{H^1}
\nc{\invlim}{\ensuremath{\displaystyle{\lim_{\longrightarrow}\fh(\hG_i;\bz)}}}
\nc{\ilim}{\ensuremath{\displaystyle{\lim \fh(\hGa_i,\bz)}}}
\nc{\WH}{\frak{H}}
\nc{\C}{\mathcal{C}}
\dmo{\aut}{Aut}
\dmo{\comm}{Comm}
\nc{\Z}{\mathbb{Z}}
\nc{\I}{\mathcal{I}}
\nc{\R}{\mathbb{R}}
\nc{\BH}{\mathbb{H}}
\nc{\Pa}{\mathcal{P}}
\nc{\TT}{\mathbb{T}}
\nc{\BS}{\mathbb{S}}
\nc{\sys}{\mathrm{sys}}
\nc{\LL}{\mathcal L}
\nc{\M}{\mathcal M}
\nc{\Flat}{\mathrm{Flat}}
\nc{\poly}{\frak{m}}
\nc{\flow}{\psi}
\nc{\A}{{\mathcal A}}
\nc{\AC}{{\mathcal A\mathcal C}}
\nc{\HH}{\mathrm{H}}
\nc{\SL}{\mathrm{SL}}
\nc{\CAT}{\rm{CAT}}
\nc{\SO}{\rm{SO}}
\nc{\Aff}{\rm{Aff}}
\nc{\Stab}{{\rm{Stab}}}
\nc{\supp}{\rm{supp}}
\nc{\Vol}{{\rm{{Vol}}}}
\nc{\shrink}{\vspace{-.2cm}}
\nc{\Area}{{\rm{{Area}}}}
\nc{\SAff}{{\rm{SAff}}}
\nc{\diam}{{\rm{diam}}}
\nc{\hull}{{\frak H}}
\nc{\ssm}{\! \smallsetminus \!}
\nc{\ee}{\varepsilon}

\title[Pseudo-Anosov subgroups of Fibered $3$--manifolds]{Pseudo-Anosov subgroups of general fibered $3$--manifold groups}
\author[Leininger]{Christopher J. Leininger}
	\address{Department of Mathematics, Rice University, Houston, TX}
	\email{cjl12@rice.edu}
\author[Russell]{Jacob Russell}
	\address{Department of Mathematics, Rice University, Houston, TX}
	\email{jr92@rice.edu}

\begin{document}

\begin{abstract} 
	We show that finitely generated and purely pseudo-Anosov subgroups of fundamental groups of fibered 3--manifolds  with reducible monodromy are convex cocompact as subgroups of the mapping class group via the Birman exact sequence.  Combined with results of Dowdall--Kent--Leininger and Kent--Leininger--Schleimer, this establishes the result for the image of all such fibered 3--manifold groups in the mapping class group.
\end{abstract}

\maketitle


\section{Introduction}
Farb and Mosher defined convex cocompactness in $\Mod(S)$---the mapping class group of a finite type orientable surface $S$ of negative Euler characteristic---via analogy with convex cocompactness of Kleinian groups \cite{FMcc}. The convex cocompact subgroups of $\Mod(S)$ play an important role in the geometry of surface group extensions and surfaces bundles \cite{FMcc,hamenstadt_hyp_ext,mjsardar} and have a rich  dynamical and geometric structure \cite{hamenstadtcc,kentleiningershadows,kentleiningeruniform,DurhamTaylor,BBKLcc}. One basic property is that convex cocompact subgroups of $\Mod(S)$ are finitely generated and \emph{purely pseudo-Anosov}, that is, every infinite order element is pseudo-Anosov. In their introductory paper, Farb and Mosher asked if this pair of properties characterized convex cocompactness \cite[Question 1.5]{FMcc}. 
A ``no'' answer  can be used to produce a relatively simple example of a finitely generated group that is not hyperbolic, but has no Baumslag--Solitar subgroups, see \cite[\S8]{kentleiningerabc}. While, a ``yes'' answer would limit the possibilities for convex compact subgroups by requiring every finitely generated subgroup of such a group to again be convex cocompact. 
We establish that the answer to Farb and Mosher's question is yes for subgroups that are contained in the image of the fundamental groups of fibered 3--manifold groups inside the mapping class group, as we now explain.

Every orientable 3--manifold that fibers over a circle is the mapping torus $M_f = S \times [0,1]/(x,1) \sim (f(x),0)$ of an orientation preserving surface homeomorphism $f \colon S \to S$. Fixing a basepoint $z \in S \subset M_f$, the fundamental group, $\Gamma_f  = \pi_1M_f =\pi_1(M_f,z)$, splits as a semi-direct product, $\Gamma_f \cong \pi_1S \rtimes_{f_*} \mathbb Z$, where $f_*$ is an automorphism of $\pi_1S=\pi_1(S,z)$ induced by $f$.
If $\phi \colon \Gamma_f \to \mathbb Z$ is the homomorphism of this splitting, then we write  $\mu \colon \mathbb Z \to  \langle f \rangle < \Mod(S)$ for the monodromy homomorphism, so that $\mu (\phi(g)) = f^{\phi(g)}$ for $g \in \Gamma_f$. Setting $S^z = S \ssm \{z\}$, the monodromy is the descent of a homomorphism $\mu^z \colon \Gamma_f \to \Mod(S^z)$, with image in the finite index subgroup $ \Mod(S^z,z) < \Mod(S^z)$ consisting of isotopy classes of homeomorphisms that fix the $z$--puncture. These homomorphisms fit into a commutative diagram with the {\em Birman exact sequence}, defined by the homomorphism $\Phi_* \colon \Mod(S^z,z) \to \Mod(S)$ induced by the inclusion $\Phi \colon S^z \to S$: 

\[
\begin{tikzcd}
1 \ar[r] 		& \pi_1S \ar[r] \ar[d, "="] 	& \Gamma_f \ar[r, "\phi"] \ar[d,"\mu^z_{\phantom{\frac..}}"] 	& \mathbb Z \ar[r] \ar[d,"\mu"] 	&1 \\
1 \ar[r] 		& \pi_1S \ar[r]				& \Mod(S^z,z) \ar[r,"\Phi_*"]								& \Mod(S) \ar[r] 			&1.\\
\end{tikzcd}
\]

The  Nielsen--Thurston Classification Theorem \cite{FLP} says that every element of a mapping class group is either pseudo-Anosov, reducible, or finite order. Our main result proves that the answer to Farb and Mosher's question is ``yes'' for subgroups of $\mu^z(\Gamma_f)$ when $f$ is an infinite order, reducible mapping class.

\begin{theorem} \label{T:reducible} Suppose $\chi(S) < 0$ and $f \colon S \to S$ is a reducible, infinite order mapping class. For any subgroup $G< \Gamma_f$, the group $\mu^z(G) <\Mod(S^z)$ is convex cocompact if and only if it is finitely generated and purely pseudo-Anosov.
\end{theorem}

\begin{remark} We note that any finitely generated, purely pseudo-Anosov subgroup $G<\Gamma_f$ as in Theorem~\ref{T:reducible} is necessarily free; see Lemma~\ref{L:free action of G}.
\end{remark}

The analogue of Theorem \ref{T:reducible} when  $f$ is pseudo-Anosov was previously shown to be true  in \cite{dowdallkentleininger}.  The analogue for $f$ finite order is a consequence of the result for subgroups $G < \pi_1S$, proved in \cite[Theorem 6.1]{kentleiningerschleimer}.  Indeed, in this case $\mu^z(\Gamma_f)$ contains $\pi_1S$ with finite index, and convex cocompactness is preserved by passage to finite index super- and subgroups.  Combining these results, we see that the conclusion holds for any $f \in \Mod(S)$.

\begin{theorem} \label{T:general} Suppose $\chi(S) < 0$ and let $f \colon S \to S$ be any mapping class. For any subgroup $G< \Gamma_f$, the group $\mu^z(G)<\Mod(S^z)$ is convex cocompact if and only if it is finitely generated and purely pseudo-Anosov.
\end{theorem}

\subsection{Known results}  There are a number of other settings where finitely generated, purely pseudo-Anosov subgroups have been shown to be convex cocompact, providing an affirmative answer to Farb and Mosher's question \cite[Question 1.5]{FMcc}. As mentioned above, if $G$ is a subgroup of either $\pi_1 S$ or $\Gamma_f$ for $f$ pseudo-Anosov, then the inclusion of $G$ into $\Mod(S^z)$ via the Birman exact sequence is convex cocompact if and only if it is finitely generated and purely pseudo-Anosov \cite{kentleiningerschleimer,dowdallkentleininger}.  The same result has been proved under the assumption that $G$ is either a subgroup of  an \emph{admissibly embedded}\footnote{It was shown in \cite{clayleiningermangahas}  that admissibly embedded right-angled Artin subgroups are quite abundant in mapping class groups.  See also \cite{Runnels}.} right-angled Artin subgroup $A < \Mod(S)$ \cite{KobMangTay} or if $G$ is contained in the genus-2 Goeritz group \cite{Goeritz_genus-2_cc}.  In \cite{DurhamTaylor}, it was shown that $G < \Mod(S)$ is convex cocompact if and only if $G$ is a stable subgroup (except for two sporadic surfaces $S$), providing a purely geometric group theoretic characterization.  This was strengthened in \cite{BBKLcc} where it was shown that $G < \Mod(S)$ is  convex cocompact if and only if $G$ is finitely generated, purely pseudo-Anosov, and undistorted.

\subsection{Proof summary} 
When $f \colon S \to S$ has infinite order, $\mu^z \colon \Gamma_f \to \Mod(S^z,z)$ is injective, and we identify $\Gamma_f$ with $\mu^z(\Gamma_f)$. For simplicity, we assume $S$ is closed in this summary, ensuring that $\Phi \colon S^z \to S$ sends every essential curve on $S^z$ to an essential curve on $S$.
To prove Theorem \ref{T:reducible}, we fix a finitely generated and purely pseudo-Anosov $G < \Gamma_f$ and show that the orbit map of $G$ to the curve complex $\C(S^z)$ is a quasi-isometric embedding.  From \cite{hamenstadtcc,kentleiningershadows}, this is equivalent to  $G$ being convex cocompact; see Theorem~\ref{T:ccc char}. The central task is then to find a way to relate distances in $G$ to distances in the curve complex. 

For subgroups $G < \pi_1S$, such a relationship was established in \cite{kentleiningerschleimer} by examining $K_u$, the stabilizer in $\pi_1 S <\Mod(S^z)$ of a simplex $u \subset \C(S^z)$. Using the isometric action of $\pi_1 S$ by deck transformations on the universal cover $p \colon \mathbb H^2 \to S$, we define $\hull_u$ to be  the convex hull of the limit set of $K_u$ in $\partial \mathbb H^2$. The group $G$ also has a convex hull for its limit set, $\hull_G$, on which it will act geometrically and which therefore serves as a geometric model for $G$.  A key argument in \cite{kentleiningerschleimer} proves that $\hull_G \cap \hull_u$ has uniformly bounded diameter, independent $u$.  The simplices that make up a geodesic edge path between $G$--orbit points in $\C(S^z)$ then give rise to a chain of bounded diameter sets in $\hull_G$. The total diameter of this chain bounds distance by a linear function of distance in $\C(S^z)$, as required.
 A similar approach is used in \cite{dowdallkentleininger} for $G < \Gamma_f$, when $f$ is a pseudo-Anosov element of $\Mod(S)$. In this case, the mapping torus $M_f$ is a hyperbolic 3--manifold, thus the convex hulls for $G$ and for simplex stabilizers can be taken in $\mathbb H^3$ instead of $\mathbb H^2$.  Once again, the key result is that these convex hulls intersect in uniformly bounded diameter sets.
 
Our proof in the reducible case is inspired by these methods. The first obstacle is that $M_f$ is not hyperbolic when $f$ is reducible, and consequently convex hulls in the universal cover are not as well-behaved.  Instead, we use the Bass--Serre tree $T$ dual to the canonical reducing multicurve $\alpha$ for $f$.  Suspending this canonical multicurve $\alpha$ in the mapping torus $M_f$ produces the torus decomposition, and $T$ is the tree dual to the tori.  The action of $\pi_1S$ on $T$ thus extends to an action of $\Gamma_f = \pi_1M_f$; see \S\ref{S:actions on H and T}.  The analogues of the hull for $G$ and for a multicurve $u \subset \C(S^z)$ are then played by a $G$--invariant subtree $T_G \subset T$ and a $K_u$--invariant subtree $T_u \subset T$, respectively; see \S\ref{S:hulls and trees}. 
Being purely pseudo-Anosov implies that $G$ acts freely on $T$, and we show that $T_G$ is a geometric model for $G$; see Lemma~\ref{L:free action of G}.  The key to proving Theorem~\ref{T:reducible} rests on showing that $T_G \cap T_u$ has bounded diameter, independent of $u \subset \C(S^z)$; see Proposition~\ref{P:bounded diameter}.
 
To understand $T_G \cap T_u$, we return to examining the convex hulls in $\mathbb H^2$.  The splitting of $\Gamma_f \cong \pi_1S \rtimes\langle f \rangle$ gives an action of $\Gamma_f$ on $\partial \mathbb{H}^2$ by homeomorphisms, extending the isometric action of $\pi_1S$ by deck transformations; see \S\ref{S:actions on H and T}.  This allows us to define $\hull_G$, which admits an isometric action by $G_0 = G \cap \pi_1S$.  Further, there is a $G_0$--equivariant inclusion $T_G \to \hull_G$, since the action on $T_G$ is free.
The quotient $p_0 \colon \hull_G \to \hull_G / G_0 = \Sigma_0$ is an infinite type, two ended surface, and $T_G/G_0 = \sigma_0$ is a spine; see \S\ref{S:G0 quotient}.  The quotient group $G/G_0 \cong \mathbb Z$ admits a cocompact, non-isometric action on $\Sigma_0$.
While the action of $G /G_0$ on $\Sigma_0$ is not isometric, the induced action on the spine $\sigma_0$ is; see \S\ref{S:polygons and G quotient}.
 
A critical technical step in our proof is the construction of a compact subsurface $\Sigma_1 \subset \Sigma_0$ so that for any simplex $u \subset \C(S^z)$, there is an element $g \in G$ for which $p_0(T_G \cap T_{g(u)}) \subset \sigma_0$ is contained in the subsurface $\Sigma_1$; see Lemma~\ref{L:translating deep}.  We say a simplex $u\subset \C(S^z)$ is {\em deep} if $p_0(T_G \cap T_{ u}) \subset \Sigma_1$, and by the previous sentence, it suffices to bound the diameter of $T_G \cap T_u$ for deep simplices.  The construction of $\Sigma_1$ is outlined in the subsection below, but we note that $\pi_1 \Sigma_1 = G_1 < G_0$ is a finitely generated, purely pseudo-Anosov subgroup.

For a deep simplex $u \subset \C(S^z)$, the intersection $\hull_G \cap \hull_u$ is contained in a uniformly bounded neighborhood of $\hull_{G_1} \cap \hull_u$.  Since $G_1 <\pi_1S$ is finitely generated and purely pseudo-Anosov, this has uniformly bounded diameter from \cite{kentleiningerschleimer}, as described above.   The vertices of the subtrees $T_G$ and $T_u$ are precisely those which are dual to regions of $\mathbb H^2 \ssm p^{-1}(\alpha)$ that $\hull_G$ and $\hull_u$ respectively intersect.  If we also knew that the vertices of $T_G \cap T_u$ were dual to regions intersected by $\hull_G \cap \hull_u$, then for deep simplices, the bound on the diameter of $\hull_G \cap \hull_u$ would imply one for $T_G \cap T_u$, and we would be done.
However, it is possible for both $\hull_G$ and $ \hull_u$ to intersect the region dual to a vertex $t \in T$, while $\hull_G \cap \hull_u$ is disjoint from it.  Our proof thus splits into two parts: bounding the diameter of a single subtree spanned by ``hull type''  vertices of $T_G \cap T_u$ that \emph{do} come from $\hull_G \cap \hull_u$, and the complementary subtrees of ``non-hull type'' vertices that do not; see \S\ref{S:subtrees}.  The former are handled as just explained; see Lemma~\ref{L:hull type subtree}.  For the latter, we proceed as follows.

For every path $\ell \subset T_G \cap T_u$ containing only ``non-hull type" vertices, we produce geodesics  $\delta_G \subseteq \partial \hull_G$ and $ \delta_u \subseteq \partial \hull_u$  that ``run parallel'' through the regions of $\mathbb H^2 \ssm p^{-1}(\alpha)$ corresponding the vertices of $\ell$. Hence we call the non-hull type vertices,  ``parallel type'' vertices.  If $\ell$ is long enough, then we show that $p_0(\delta_G)\subset \Sigma_1$ must project to a closed boundary geodesic of $\Sigma_1$.  Since this geodesic represents an element of $G_1$, it is pseudo-Anosov, and so further projects to a filling geodesic in $S$ by a result of Kra \cite{kra}; see Theorem~\ref{T:Kra}.  On the other hand, $\delta_u$ projects to a simple closed geodesic in $S$ (isotopic to a component of $\Phi(u)$; see \S\ref{S:fibers and trees}).  But if $\ell$ is too long, then the simple closed geodesic image of $\delta_u$ runs parallel to the filling geodesic image of $\delta_G$ for a long time, which is a contradiction.  This proves a uniform bound on the diameter of $\ell$; see Lemma~\ref{L:bounding parallel subtrees}.  Combining the hull-type and parallel-type subtree bounds for deep simplices, proves a uniform bound on $\diam(T_G \cap T_u)$, for every simplex $u \subset \C(S^z)$, as required.

     \subsection{Construction of $\Sigma_1$} We now outline the construction of the compact subsurface $\Sigma_1 \subset \Sigma_0$ with the property that for every simplex $u \subset \C(S^z)$, the image $p_0(T_G \cap T_u)$ can be translated into $\Sigma_1$ by an element of $G$.
Since the spine $\sigma_0 = p_0(T_G)$ of $\Sigma_0$ has an isometric action of $\mathbb Z \cong G/G_0$, it suffices to prove a uniform bound on $p_0(T_G \cap T_u)$ in $\sigma_0$:  we then simply take a sufficiently large neighborhood $\sigma_1$ of a fundamental domain for the action of $G/G_0$ on the spine $\sigma_0 \subset \Sigma_0$, and take $\Sigma_1$ to be a thickening of $\sigma_1$ in $\Sigma_0$.

To bound the diameter of  $p_0(T_G \cap T_u)$ in $\sigma_0$, we utilize Masur and Minsky's subsurface projections \cite{MM2}. For simplicity we describe the idea in the case where our reducible surface homeomorphism $f \colon S \to S$ is a Dehn twist about a single curve $\alpha$. First, we let $A \to S$ be the annular cover whose core curve is $\alpha$, and for every simplex $u \subset \C(S^z)$, let $\pi(v)$ be the subsurface projection of $v=\Phi(u) \subset \C(S)$ to the {\em arc graph} of $A$; see \S\ref{S:MM projections}.  Next, for every edge $e \subset T_G$, there is a dual geodesic $\widetilde \alpha_e \subset p^{-1}(\alpha)$, and we can identify the annulus $A$ with the quotient $A = \mathbb H^2/\Stab_{\pi_1S}(\widetilde \alpha_e)$.   There are two boundary components of $\partial \hull_G$ that non-trivially intersect $\widetilde \alpha_e$, and we let $\Delta_e$ denote their image in $A$, viewed as a subset of the arc graph of $A$; see \S\ref{S:projections}.
These sets $\Delta_e$ decorate the edges $e$ of $T_G$ and are $G$--equivariant, with $\Delta_{g(e)} = f^{\phi(g)}  (\Delta_e)$ where $f^{\phi(g)}$ is the image of $g \in G$ under the homomorphism $G \to \langle f \rangle$; see Lemma~\ref{L:equivariant decoration}.

     The key idea is now the following: any edge $e$ for which the intersection number of an arc in $\Delta_e$ with one from $\pi(v)$ is sufficiently large is a ``dead end", beyond which the hull type subtree and any parallel type subtree cannot extend; see Lemma~\ref{L:small_proj_parallel_type} and Lemma~\ref{L:dead ends}.
Since $\sigma_0$ has finite valence, distant vertices in $\sigma_0$ should basically differ by a large power of the generator of $G / G_0$, which is isotopic to a large power of the Dehn twist $f$. Thus for a geodesic $\ell$ in $T_G \cap T_u$ with $p_0(\ell)$ sufficiently long,  there must be two edges $e,e'$ of $\ell$  and an element of $g\in G$ so that $e' = g(e)$ and $|\phi(g)|$ is  large.  However, if  $\Delta_e$ and $\pi(v)$ have small intersection number (because $e$ is not a ``dead end''), then  $\Delta_{e'} = \Delta_{g(e)} = f^{\phi(g)} (\Delta_e)$ and $\pi(v)$ will have large intersection number (depending on $|\phi(g)|$). Thus we get a bound on how large $|\phi(g)|$ can be, and hence a bound on how long $p_0(\ell)$ can be.
     
 Our proof in the general case of arbitrary reducible $f$ follows the same basic idea using the subsurface projection to the complementary components of $S \ssm \alpha$ to give a decoration on the vertices of $T_G$ in addition to decorations of edges coming from the annular covers.  The argument is complicated by the fact that $f$ may act trivially on some subsurfaces and some annuli; see \S\ref{S:Bounding_intersection_in_sigma_0} for details.

\medskip

\noindent
{\bf Acknowledgments.} The first author thanks Spencer Dowdall, Autumn Kent, and Saul Schleimer for their collaboration on \cite{kentleiningerschleimer} and \cite{dowdallkentleininger}, which served as motivation for the current work.   Both authors thank Jason Behrstock and Dan Margalit for comments on an earlier draft of the paper.  The first author was supported by NSF grant DMS-2106419 and the second by   NSF grant DMS-2103191.

     \section{Preliminaries}
     
     Throughout, $S$ will denote a connected orientable finite type surface with negative Euler characteristic. We will equip this surface with a complete hyperbolic metric of finite area that identifies $\mathbb H^2$ with the universal cover $p \colon \mathbb H^2 \to S$.  Given a point $z \in S$, we let $S^z$ denote the surface obtained by puncturing $S$ at $z$. We also equip $S^z$ with a complete, finite area hyperbolic metric. The \emph{curve complex} of $S$ (or $S^z$) is the flag simplicial complex $\C(S)$ whose vertices are isotopy classes of essential, simple closed curves on $S$ with two isotopy classes joined by an edge if they have disjoint representatives. Each vertex of $\C(S)$ has a unique geodesic representative and two vertices will be joined by an edge if and only if these geodesic representative are disjoint. Hence, each simplex of $\C(S)$ corresponds to a multicurve on $S$, which has a unique geodesic representative.  Whenever convenient, we will assume that a simplex/multicurve $v \subset \C(S)$ is represented in $S$ as a geodesic multicurve.
     
     Given a surface with boundary $Y$, we define the \emph{arc and curve complex}  to be the flag simplicial complex $\AC(Y)$ whose vertices are isotopy classes of both essential, simple closed curves and isotopy classes of essential arcs meeting the boundary of $Y$ precisely in their endpoints.\footnote{One often allows properly embedded arcs with ends in cusps of $Y$, if any, but we will not need such arcs in our work, so omit them in our definition.}  As with the curve complex, two vertices of $\AC(Y)$ are joined by an edge if there are disjoint representatives for the isotopy classes.  
     
 When $S$ is a once-punctured torus or four-punctured sphere, one usually makes a different definition for $\C(S)$, but we do not do that here.  In particular, these curve complexes are discrete, countable sets.  On the other hand, if $Y$ is a torus with one boundary component or a sphere with at least one boundary component and the sum of the number of boundary components and punctures equal to $4$, then we do take the usual modified definition for $\AC(Y)$ in which vertices are joined by an edge if they intersect once or twice for these two types of surfaces, respectively.  The reason is that for $\C(S)$, we need Theorem~\ref{T:fibers_are_trees} below to hold, while for $\AC(Y)$, we will use coarse geometric properties in \S\ref{S:Bounding_intersection_in_sigma_0}.

     If $A \to S$ is an annular cover, let $\bar A$ denote the compact annulus obtained from $A$ by adding its ideal boundary from the hyperbolic metric on $S$.  This compactification, $\bar A$, of $A$ is independent of the choice of metric.  The \emph{arc complex} $\A(A)$ is the flag simplicial complex whose vertices are isotopy classes of essential arcs on $\bar A$, where unlike other surfaces with boundary, isotopies of $\bar A$ are required to be the identity on $\partial \bar A$.  Edges of $\A(A)$ correspond to pairs of isotopy classes with representatives having disjoint interiors.  The annuli of primary interest come from curves $w \in \C(S)$.  More precisely, every such curve $w$ determines a conjugacy class of cyclic subgroups of $\pi_1S$ and hence an annular covering (unique up to isomorphism) $A = A_w \to S$ for which $w$ lifts to the core curve.

     \subsection{Mapping class groups and Birman exact sequence} \label{S:MCG-Birman}
     
     We recall that the mapping class group of $S$ is the group of orientation preserving homeomorphisms (or diffeomorphisms) of $S$, modulo the normal subgroup of those homeomorphisms that are isotopic to the identity,
     \[\Mod(S) = \Homeo^+(S)/\Homeo_0(S).\]
     Every element of $\Mod(S)$ is thus the isotopy class of a homeomorphism.
     
     Recall that we have fixed a basepoint $z \in S$, and $S^z = S \ssm \{z\}$. We write $\Phi \colon S^z \to S$ for the inclusion.  The puncture of $S^z$ that accumulates on $z$ via $\Phi$ is called the \emph{$z$--puncture} and we often refer to $\Phi$ as the map that ``fills the $z$--puncture back in".
     
     We are interested in the finite index subgroup $\Mod(S^z,z) < \Mod(S^z)$ consisting of isotopy classes of homeomorphisms that fix the $z$--puncture.  Any homeomorphism $\varphi \colon S^z \to S^z$ defining an element of $\Mod(S^z,z)$ uniquely determines a homeomorphism $\varphi' \colon S \to S$ extending over the point $z$ by sending it to itself and by the formula $\varphi' \circ \Phi = \Phi \circ \varphi$ on $S^z$.  When the context makes the meaning clear, we usually abuse notation and use the same symbol $\varphi$ to denote the mapping class in $\Mod(S^z,z)$, a representative homeomorphism of $S^z$, as well as the unique extension to a homeomorphism of $S$.
     
     The extension of a homeomorphism of $(S^z,z)$ over the point $z$ via the map $\Phi$ defines a surjective homomorphism $\Phi_* \colon \Mod(S^z,z) \to \Mod(S)$, and the \textit{Birman's exact sequence} \cite{birmanbraids} identifies an isomorphism of the kernel of $\Phi_*$ with $\pi_1S$:
     
     \begin{equation*} \label{E:BES}
     	\begin{tikzcd}
     		1 \ar[r] 		& \pi_1S \ar[r]				& \Mod(S^z,z) \ar[r,"\Phi_*"]								& \Mod(S) \ar[r] 			&1.
     	\end{tikzcd}
     \end{equation*}
    
     It will be useful to describe explicitly the isomorphism of the kernel of $\Phi_*$ with $\pi_1S$. 
     If $\varphi \colon S^z \to S^z$ represents an element of the kernel, then the extension $\varphi \colon S \to S$ over the point $z$ is isotopic to the identity, by an isotopy that does {\em not preserve} $z$.  If $\varphi_t \colon S \to S$ is the isotopy so that $\varphi_0 = \varphi$ and $\varphi_1 = {\bf 1}_S$, then defining $\gamma(t) = \varphi_t(z)$, we see that $\gamma$ is a loop based at $z$.  The isomorphism of the kernel with $\pi_1S$ assigns the homotopy class of $\gamma$ to $\varphi \in \Mod(S^z,z)$.  Alternatively, we can think of producing a homeomorphism $\varphi \colon S^z \to S^z$ by {\em pushing} $z$ around the loop $\gamma^{-1}$ by an isotopy on $S$; we call this the {\em point push around $\gamma^{-1}$}.
     
     Another perspective is useful in our setting.  Fix a point $\widetilde z \in p^{-1}(z)$.   Any homeomorphism $\varphi \colon S^z \to S^z$ (representing an element of $\Mod(S^z,z)$) has a unique lift $\widetilde \varphi \colon \mathbb H^2 \to \mathbb H^2$ fixing $\widetilde z$.  The lift $\widetilde \varphi$ is a quasi-isometry,\footnote{When $S$ has cusps we assume any homeomorphism $\varphi$ of $S$ is an isometry in some neighborhood of the cusps; this is a convenience, however, as the extension of the lift to $\partial \mathbb H^2 \to \partial \mathbb H^2$ exists independent of this assumption.}
     and so has a unique extension to a homeomorphism $\partial \mathbb H^2 \to \partial \mathbb H^2$. Any other representative of the isotopy class of $\varphi$ in $\Mod(S^z,z)$ has the same extension, since the lift of the isotopy moves all points a bounded hyperbolic distance, thus we obtain an action of $\Mod(S^z,z)$ on $\partial \mathbb H^2$.
     
     Next observe that if $\varphi_0 \colon S^z \to S^z$ represents an element in the kernel of $\Phi_*$, and $\varphi_t \colon S \to S$ is the isotopy to the identity.  This isotopy lifts to an isotopy $\widetilde \varphi_t$ from the lift $\widetilde \varphi_0$ fixing $\widetilde z$ to a lift of the identity.  The resulting lift of the identity,   $\widetilde \varphi_1$, is thus a covering transformation, namely the one associated to $\gamma$ (as defined above by $\gamma(t) = \varphi_t(z)$).  Thus, we have the following:
     \begin{proposition}[{c.f.~\cite[\S1.2.3]{leiningermjschleimer}}] \label{P:extensions} The restriction of the action of $\Mod(S^z,z)$ on $\partial \mathbb H^2$ to $\pi_1(S)$ agrees with the extension of the isometric covering action of $\pi_1S$ on $\mathbb H^2$.
     \end{proposition}
     
     Kra's Theorem \cite{kra} describes precisely which elements of $\pi_1S$ represent pseudo-Anosov elements of $\Mod(S^z,z)$.  Recall that a loop is {\em filling} if it cannot be homotoped disjoint from any essential simple closed curve (and is thus a property of the homotopy class).
     \begin{theorem}[\cite{kra}] \label{T:Kra} An element of $\pi_1S$ represents a pseudo-Anosov element of $\Mod(S^z,z)$ if and only if it is represented by a filling loop.
     \end{theorem}
     Since being pseudo-Anosov is equivalent to not having any isotopy classes of periodic simple closed curves, the point pushing description of Birman's isomorphism suggests a proof of Theorem~\ref{T:Kra}; see \cite[\S14.1.4]{farbmargalit}.
     
     \subsection{Fibers and trees} \label{S:fibers and trees}

     We let $\C^s(S^z) \subset \C(S^z)$ denote the subcomplex spanned by curves whose image under $\Phi \colon S^z \to S$ is an essential curve on $S$.  We call the vertices of $\C^s(S^z)$ the \emph{surviving curves} of $S^z$. Since $\Phi$ maps disjoint curves  to disjoint curves, it induces a simplicial, surjective map which we also denote $\Phi \colon \C^s(S^z) \to \C(S)$, by an abuse of notation. Given any simplex, $v \subset \C(S)$, we let $\Phi^{-1}(v)$ denote the preimage of the barycenter of $v$.  The following is proved in \cite{kentleiningerschleimer}.
     
     \begin{theorem}\label{T:fibers_are_trees} For any simplex $v \subset \C(S)$, there is a $\pi_1S$--equivariant homeomorphism from the Bass--Serre tree $T$ dual to $v$ to $\Phi^{-1}(v) \subset \C^s(S^z)$. The image of a vertex $t \in T$ under this homeomorphism is the barycenter of a simplex $u_t \subset \C^s(S^z)$ for which $\Phi(u_t) = v$ and $\Phi \vert_{u_t}$ is injective. Moreover, $t,t'\in T$ are joined by an edge if and only if $u_t \cup u_{t'}$ spans a simplex of $\C^s(S^z)$.
     \end{theorem}
     
     The proof of Theorem~\ref{T:fibers_are_trees} involves some ideas that will be useful for us, which we briefly describe.  Given a simplex $u \subset \C(S^z)$, we let $K_u$ denote the stabilizer of $u$ in $\pi_1S < \Mod(S^z,z)$ and $\hull_u \subset \mathbb H^2$ denote the convex hull of the limit set of $K_u$  in $\partial \mathbb H^2$ (if it is nonempty).  If $u \subset \C^s(S^z)$, $v = \Phi(u)$, and $\Phi|_u$ is injective, then $p \colon \mathbb H^2 \to S$ maps the interior $\hull_u^\circ \subset \hull_u$ to a component of $S \ssm v$ (where $v$ is realized by its geodesic representative).
     Up to isotopy, $p(\hull_u^\circ)$ is the $\Phi$--image of the component $U \subset S^z\ssm u$ containing the $z$--puncture.  One way to think about this fact is that point pushing around a loop preserves $u$ precisely when the loop is disjoint from $u$, that is, when the loop (intersected with $S^z$) is contained in $U$.  When $\Phi|_u$ is not injective,  the component of $S^z \ssm u$  containing the $z$--puncture is a once-punctured annulus, making $K_u$ an infinite cyclic group.  In any case, the stabilizer of $\hull_u$ is exactly $K_u$; see \cite[Theorem~4.1]{kentleiningerschleimer}.

     \subsection{Convex cocompactness}
     
     Farb and Mosher originally defined convex compactness in the mapping class group using the action on Teichm\"uller space; see \cite{FMcc}. For our purposes, it will be most convenient to use the following formulation due to Kent--Leininger  and independently Hamenst\"adt.
     
     \begin{theorem}[\cite{kentleiningershadows,hamenstadtcc}] \label{T:ccc char} A subgroup of the mapping class group is convex cocompact if and only if it is finitely generated and the orbit map to the curve complex is a quasi-isometric embedding.
     \end{theorem}
     
     We will apply this to the case of subgroups of $\Mod(S^z)$.  We note that since the inclusion of a finite index subgroup into a bigger group is a quasi-isometry, convex cocompactness survives passage between finite index super- and subgroups.

\section{Set up} \label{S:set up}

We now fix a homeomorphism $f \colon S \to S$ that defines an infinite order, reducible mapping class in $\Mod(S)$.  We let $\Gamma$ denote the $\pi_1 S$--extension group $\Gamma_f$ that is the fundamental group of the mapping torus for $f$. Since $f$ defines an infinite order mapping class, the homomorphism $\mu^z \colon \Gamma \to \Mod(S^z,z)$ is injective and we identify $\Gamma$ with its image in $\Mod(S^z,z)$. Let $\alpha = \alpha_1\cup \ldots \cup \alpha_n \subset S$ be the canonical reduction system for the reducible mapping class defined by $f$; see \cite{ivanov}. 
Since convex cocompactness is preserved by passing to finite index super-groups and $\Gamma_{f^n} < \Gamma_f$ has finite index, we can replace $f$ with a power when it is helpful.  We do so, and thus (after an isotopy if necessary) assume that $f$ fixes each curve $\alpha_i$ and each component of $S \ssm \alpha$. By possibly raising to a further higher power, we can also assume that $f$ restricted to each component of $S\ssm \alpha$ is either the identity or pseudo-Anosov.
We also assume throughout that $\alpha$ is realized as a geodesic multicurve in $S$ with respect to our fixed hyperbolic metric.

A {\em complementary subsurface} of $\alpha$ is defined as the path metric completion $Y$ of a component $Y^\circ \subset S \ssm \alpha$.  Such a complementary subsurface $Y$ is a hyperbolic surface with geodesic boundary and the inclusion $Y^\circ \to S \ssm \alpha$ extends to an immersion $Y \to S$ that is injective on the interior, and at most $2$-to-$1$ at points of $\partial Y$.  By an abuse of notation, we often write $Y \subset S$ or refer to the map $Y \to S$ as the inclusion.

Write $Y_1,\ldots,Y_k$ to denote the complementary subsurfaces of $\alpha$.
Since each $\alpha_i$ and each $Y_j^\circ$ is invariant by $f$,  we obtain ``restricted" maps $f|_{Y_j} \colon Y_j \to Y_j$.   We can re-index the complementary subsurfaces so that there is some $0 \leq m \leq k$ such that for $j \leq m$, $f\vert_{Y_j}$ is pseudo-Anosov on $Y_j$ and for $j > m$, $f\vert_{Y_j}$ is the identity.   We refer to these subsurfaces $Y_j$ as the {\em pseudo-Anosov components} and the {\em identity components}, respectively.

 Given $f$ as above, we fix a finitely generated subgroup $G <\Gamma$ that is purely pseudo-Anosov as a subgroup of $\Mod(S^z)$. If $G$ is in the kernel of the homomorphism $\Phi_* \colon \Gamma \to \langle f \rangle <\Mod(S)$, then $G$ is contained in $\pi_1S$, and hence is convex cocompact in $\Mod(S^z)$ by \cite[Theorem~6.1]{kentleiningerschleimer}.  Thus, we may assume that $\Phi_*$ sends $G$ onto the subgroup of $\langle f \rangle $ generated by $f^n$ for some $n > 0$.  By passing to a further power if necessary, we can assume that $n=1$; this will be convenient for notational purposes later and does not affect any other properties of $f$ we have already assumed.  If $\phi|_G$ is injective, then $G \cong \mathbb Z$, and the theorem also follows, so we assume $G_0 = \ker(\phi|_G) = G \cap \pi_1S$ is nontrivial (hence infinite).

\subsection{Action on $\mathbb H^2$ and $T$.} \label{S:actions on H and T}

We assume that in our fixed hyperbolic metric, the lengths of the components of $\alpha$ are short enough that any two components of $p^{-1}(\alpha)$ are distance at least $2$ apart in $\mathbb H^2$.  
The action of $\pi_1 S$ on $\mathbb H^2$ preserves $p^{-1}(\alpha)$, and we let $T$ denote  Bass--Serre tree dual to $p^{-1}(\alpha)$. 
The action of $\pi_1S$ on $\mathbb H^2$ and $T$ extend to an action of $\Gamma$ as we now explain.

As in \S\ref{S:MCG-Birman}, given any $\varphi \in \Gamma$ we write $\varphi \colon S^z \to S^z$ for a representative and its extension $\varphi \colon S \to S$ (after filling the $z$--puncture back in). Since $\varphi \in\Gamma$, the homeomorphism $\varphi \colon S \to S$ is isotopic (ignoring $z$) to $f^k$, for some $k \in \mathbb Z$.    The lift $\widetilde \varphi \colon \mathbb H^2 \to \mathbb H^2$ fixing $\widetilde z$ is thus isotopic to a lift of $f^k$ (not necessarily fixing $\widetilde z$), and so has the same extension to $\partial \mathbb H^2$.  
Given any lift $\widetilde f \colon \mathbb H^2 \to \mathbb H^2$ of $f$, any lift of $f^k$ is then obtained by composing $\widetilde f^k$ with an element of $\pi_1S$. Conversely, any such composition is a lift of $f^k$.  Hence, the action of $\Gamma$ on $\partial \mathbb H^2$ factors through an isomorphism with the group $\langle \widetilde f, \pi_1S \rangle$ acting on $\partial \mathbb H^2$.
  This isomorphism $\Gamma \cong \langle \widetilde f,\pi_1S \rangle$ then defines an action on $\mathbb H^2$ extending the covering action of $\pi_1S$.
Alternatively, the given lift $\widetilde f$ is equivariantly isotopic to the lift $\widetilde \varphi$ of some $\varphi \in \Gamma$ with $\Phi_*(\varphi) = f$.  Then $\Gamma = \pi_1S \rtimes \langle \varphi \rangle$ acts on $\mathbb H^2$ so that $\pi_1S$ acts by covering transformations and $\varphi^k$ acts by $\widetilde f^k$ for all $k \in \mathbb Z$. 

 An alternative way to see the action of $\Gamma$ on $\mathbb H^2$ is to consider the universal covering $\widetilde M_f$ of the mapping torus, lift the suspension flow, and consider the quotient by the flow lines.  Since the universal cover $\mathbb H^2$ of $S$ intersects each lifted flow line once, the flow space is identified with $\mathbb H^2$ and the action of $\Gamma$ on $\widetilde M_f$ descends to an action of $\Gamma$ on $\mathbb H^2$ which agrees with the covering action when restricted to $\pi_1S$.

Since $f$ preserves $\alpha$, $\widetilde f$ preserves $p^{-1}(\alpha)$.  Therefore, $\Gamma$ acts on the Bass--Serre tree $T$ dual to $\alpha$. Since $f$ fixes each $Y_i$ and each curve in $\alpha$, a pair of vertices/edges of $T$ are in the same $\Gamma$--orbit if and only if they are in the same $\pi_1 S$--orbit. Unlike the action on $\mathbb H^2$, this action on $T$ is by isometries.
For each edge $e \subset T$,  we write $\widetilde{\alpha}_e$ to denote the component of $p^{-1}(\alpha)$ that is dual to $e$.  We choose a $\Gamma$--equivariant map $\mathbb H^2 \to T$ sending $\widetilde{\alpha}_e$ to the midpoint of $e$, and each component of $S \ssm p^{-1}(\alpha)$ to the $\frac12$--neighborhood of the dual vertex.  There are many such choices, and we sometimes make a choice of one that is convenient for certain applications; for example, we may take such a map to be $K$--Lipschitz, where $K$ depends only on the minimal distance between pairs of components of $p^{-1}(\alpha)$.  We also choose a $\pi_1S$--equivariant map $T \to \Phi^{-1}(\alpha) \subset \C^s(S^z) \subset \C(S^z)$ as in Theorem~\ref{T:fibers_are_trees}, identifying vertices and edges of $T$ with simplices of $\C(S^z)$ in $\Phi^{-1}(\alpha)$.

\subsection{Subsurfaces and annuli for vertices and edges}

For each vertex $t \in T$, we let $\widetilde{Y}_t^\circ$ denote the component of $ \mathbb H^2 \ssm p^{-1}(\alpha)$ dual to $t$, and use $\widetilde Y_t$ for its closure.  Let $K_t= \Stab_{\pi_1S}(\widetilde{Y}_t)$ and define $Y_t$ to be $\widetilde Y_t/K_t$.  We can identify each $Y_t$ with exactly one of the complementary subsurfaces $Y_1,\dots,Y_k$ as follows:
 for a vertex $t \in T$, let $\Upsilon_t = \mathbb H^2 / K_t$. The surface $Y_t$ is then the convex core of $\Upsilon_t$ and there is a unique $i(t) \in \{1,\dots,k\}$ so that the covering map $\Upsilon_t \to S$ maps the interior of $Y_t$ isometrically onto $Y_{i(t)}^\circ \in \{Y_1^{\circ},\dots,Y_k^\circ\}$. If $t$ and $t'$ are in the same $\Gamma$--orbit, then  $\Upsilon_t$ and $\Upsilon_{t'}$ are equivalent covers of $S$ with different choices of base point. Hence there is an isomorphism of covering spaces $\Upsilon_t \to \Upsilon_{t'}$ that sends $Y_t$ isometrically to $Y_{t'}$. In particular, $Y_{i(t)} = Y_{i(t')}$ and we use this to identify $Y_t = Y_{i(t)} = Y_{i(t')} = Y_{t'}$.

For each edge $e \subset T$, we let $K_e = \Stab_{\pi_1S}(\widetilde \alpha_e)$ and define $A_e$ to be the annulus $\mathbb H^2 / K_e$.   There exist a unique $i(e)\in \{1,\dots,n\}$ so that $p(\widetilde \alpha_e) = \alpha_{i(e)}$.  When convenient, we will also write $\alpha_e = \alpha_{i(e)}$.  When $e$ and $e'$ are in the same $\Gamma$--orbit, $A_e$ and $A_{e'}$ are equivalent annular covers of $S$ with core curve $\alpha_e$. Hence, we can isometrically identify all these annuli: $A_e = A_{i(e)} = A_{i(e')} =A_{e'}$.

We note that each vertex $t$ and edge $e$ of $T$ is identified with simplices $a_t$ and $a_e$ of $\C^s(S^z)$, respectively, by Theorem~\ref{T:fibers_are_trees}, and $K_t = K_{a_t}$ and $K_e = K_{a_e}$ are indeed special cases of simplex stabilizers (so the notation is compatible with that in \S\ref{S:fibers and trees}).  Moreover, $\widetilde Y_t = \hull_{a_t}$ and $\widetilde \alpha_e = \hull_{a_e}$, as in \S\ref{S:fibers and trees}.  Using this, and the fact that the $\Gamma$--orbits and $\pi_1S$--orbits of vertices and edges of $T$ are the same, it follows that the $\pi_1S$--equivariant map $T \to \Phi^{-1}(\alpha) \subset \C^s(S^z) \subset \C(S^z)$ is also $\Gamma$--equivariant.

	\subsection{Hulls and Trees} \label{S:hulls and trees}

	We now define invariant subtrees of the Bass--Serre tree $T$ for the simplex stabilizers $K_u$ as well as our purely pseudo-Anosov subgroup $G < \Gamma$. These subtrees will allow us to translate distances in $\C(S^z)$ to distances in $G$.
	
	For each simplex $u \subset \C(S^z)$, the stabilizer $K_u <\pi_1 S$ acts by isometries on $T$. If the action of $K_u$ does not have a global fixed point, we let $T_u$ be the minimal invariant subtree  of $K_u$. In this case, $T_u$ is the union of the axes of loxodromic elements;  see, e.g.~\cite[Proposition~2.9]{Bestvina-Trees}.    If $K_u$ has a global fixed point in $T$, we define $T_u$ to be the maximal fixed subtree. We can readily determine  the structure of $T_u$ by examining the component of $S^z \ssm u$ that contains the $z$--puncture.
	
	\begin{lemma}\label{L:options_for_T_u}
		Let $u \subset \C(S^z)$ be a multicurve and $U$ be the component of $S^z \ssm u$  that contains the $z$--puncture.  
		\begin{enumerate}
			\item 	The action of $K_u$ on $T$ has a global fixed point if and only if $\alpha$ can be isotoped to be disjoint from $\Phi(U)$ in $S$. \label{I:global_fixed_point}
			\item When $K_u$ has a global fixed point, $T_u$ is either a single vertex $t \in T$ or a single edge $e \subset T$. Moreover, $T_u$ is an edge $e$ if and only if $U$ is a once-punctured annulus and each component of $\Phi(\partial U)$ is isotopic to the curve $\alpha_e$ of $\alpha$. \label{I:fixed_sets}
			\item If $u$ contains a non-surviving curve, then $T_u$ is a single vertex. \label{I:non-surviving}
			\item When $u$ consist only of surviving curves and $T_u$ is not an edge, then $t \in T_u$ if and only if $\hull_u \cap \widetilde Y_t^\circ \neq \emptyset$. \label{I:T_hull=H-hull}
		\end{enumerate}
	\end{lemma}

	\begin{proof}
		As described in \S \ref{S:fibers and trees}, $K_u$ is the group of all pushes along loops in $U$ based at $z$ (after filling the $z$--puncture back in). This group is naturally isomorphic to $\pi_1(\Phi(U), z) <\pi_1(S,z) = \pi_1S$.  Hence $K_u$ contains a hyperbolic isometry of $\pi_1S$ if and only if $\Phi(U)$ is not a once-punctured disk.
		
		Now observe that $\Phi(U)$ is a once-punctured disk if and only if $u$ contains a non-surviving curve.  In this case, $K_u$ is an infinite cyclic group generated by a parabolic isometry. Hence, there is an invariant horoball for $K_u$ that is contained in $\widetilde Y^\circ_t$ for some vertex $t\in T$. It follows that $K_u$ fixes no geodesic in $p^{-1}(\alpha)$, but fixes  $\widetilde Y^\circ_t$.  This implies $T_u = \{t\}$, which proves part \eqref{I:non-surviving}.

		We now focus on the case where $\Phi(U)$ is not a once-punctured disk, so $K_u$ contains a hyperbolic isometry of $\pi_1 S$, or equivalently, when $\hull_u$ is non-empty.
		
		The fixed points in $T$ of the hyperbolic elements of $\pi_1 S$ are determined by their axes in $\mathbb H^2$ as follows. Let $g \in \pi_1S$ be hyperbolic and let $\gamma_g$ be the axis of $g$ in $\mathbb H^2$.  If $\gamma_g\subset\widetilde Y_t^\circ$ for some vertex  $t \in T$, then $t$ is the unique fixed point of $g$. If $\gamma _g = \widetilde \alpha_e$ for some geodesic  $\widetilde \alpha_e \subset p^{-1}(\alpha)$, then the edge $e \subset T$ is the maximal fixed subtree of $g$ in $T$. Finally, if $\gamma_g$  crosses a geodesic in $p^{-1}(\alpha)$, then periodicity says the set of geodesics in  $p^{-1}(\alpha)$ that $\gamma_g$ crosses will be the edges of a bi-infinite geodesic  $\ell_g \subset T$. In this case, $g$ acts loxodromically on $T$, and its axis in $T$ is $\ell_g$. Conversely, if $g$ acts loxodromically on $T$ with axis $\ell_g \subset T$, then $\gamma_g$ crosses all the geodesics in $p^{-1}(\alpha)$ corresponding to the edges of $\ell_g$. 
		
		Returning to $K_u$, if $\hull_u$ contains a bi-infinite geodesic that intersects a geodesic in $p^{-1}(\alpha)$ transversely, then there is some hyperbolic element of $K_u$ whose $\mathbb{H}^2$--axis crosses a geodesic in $p^{-1}(\alpha)$. Since this element will act loxodromically on $T$, $K_u$ has no global fixed point, and thus $T_u$ is the union of axes of the elements of $K_u$ that act loxodromically on $T$.  A vertex $t \in T$ is then on the $T$--axis of a loxodromic element $g \in K_u$ if and only if $g$ is a hyperbolic element of $\pi_1 S$ whose $\mathbb H^2$--axis intersects $\widetilde Y_t$ in a bounded diameter segment. Thus, $t \in T_u$ if and only if $\hull_u \cap \widetilde Y_t^\circ \neq \emptyset$, proving part \eqref{I:T_hull=H-hull}.  In this case, $p(\hull_u^\circ)$ is isotopic to $\Phi(U)$, so $\hull_u$ containing a geodesic that intersects a geodesic in $p^{-1}(\alpha)$ transversely ensures that $\alpha$ cannot be isotoped to be disjoint from $\Phi(U)$, proving (the contrapositive of) one of the implications in part \eqref{I:global_fixed_point}.
		
		If $\hull_u$  does not contain a geodesic that intersects a geodesic in $p^{-1}(\alpha)$ transversely, then either $\hull_u^\circ$ is contained entirely in $\widetilde Y_t^\circ$ for some vertex $t \in T$, or $\hull_u$ is one of the geodesics $\widetilde \alpha_e \subset p^{-1}(\alpha)$.  If $\hull_u^\circ \subset \widetilde Y_t^\circ$, then $T_u = \{t\}$. In this case,  $p(\hull_u^\circ)$ is isotopic to $\Phi(U)$, so  $\hull_u^\circ \subset \widetilde Y_t^\circ$ implies $\Phi(U)$ is disjoint from $\alpha$.
If $\hull_u = \widetilde \alpha_e$, then $K_u = \Stab_{\pi_1S}(\hull_u)$ is an infinite cyclic group generated by a hyperbolic isometry whose axis is the geodesic $\widetilde \alpha_e$. Thus, $T_u = e$ and $p(\hull_u) = \alpha_e \subset \alpha$.  It follows that $\Phi(\partial U)$ is an annulus with core curve isotopic to $\alpha_e$. Thus, $U$ is a once-punctured annulus and $\alpha$ can be isotoped to be disjoint from $\Phi(U)$.  Combined with the case where $\Phi(U)$ is a once-punctured disk, this proves the other implication of part \eqref{I:global_fixed_point}.  Furthermore, when combined with the discussion above from the proof of part \eqref{I:non-surviving}, we also deduce part \eqref{I:fixed_sets}.  This completes the proof of the lemma.
	\end{proof}

	These invariant subtrees have the following intersection property for nested simplices. This allows us to produce paths in $T$ from paths in $\C(S^z)$.

	\begin{lemma}\label{L:nested_simplices}
		Let $u,w$ be simplices of $\C(S^z)$. If $u \subseteq w$, then $T_u \cap T_w \neq \emptyset$.	
	\end{lemma}
	
	\begin{proof}
		Since $u \subseteq w$, we have $K_w < K_u$. 
		If $K_u$ has a global fixed point, then $T_u$ is the maximal fixed subtree of $K_u$, and hence $T_w$ is also the maximal fixed subtree of $K_w$.  In this case, $T_u \subseteq T_w$.   If neither has a global fixed point, then $T_u$ and $T_w$ are the minimal invariant subtrees of $K_u$ and $K_w$, respectively, and so $T_w \subseteq T_u$.  In either of these cases, $T_u \cap T_w \neq \emptyset$.
		
		Finally, suppose $T_u$ is a minimal invariant subtree for $K_u$ and $T_w$ is the maximal fixed subtree of $K_w$.  By Lemma~\ref{L:options_for_T_u}, either $T_w$ is a vertex or edge, and in either case, there is an element $g \in K_w$ whose fixed point set is exactly $T_w$.  Since $T_u$ has no global fixed point, there is an axis $\ell \subset T_u$ for an element $h \in K_u$ acting loxodromically on $T_u$.  If $\ell \cap T_w \neq \emptyset$, then $T_u \cap T_w \neq \emptyset$, as required.   On the other hand, if $\ell \cap T_w = \emptyset$, then $\ell \cap g (\ell) = \emptyset$, and the geodesic from $\ell$ to $g (\ell)$ must non-trivially intersect $T_w$.  Since this geodesic is contained in $T_u$, it follows that $T_u \cap T_w \neq \emptyset$. 
	\end{proof}
	
	Recall that we have fixed a finitely generated and purely pseudo-Anosov subgroup $G < \Gamma$ and have passed to an appropriate power of $f$ so that $\Phi_*(G) = \langle f \rangle$.  We have $G_0 = G \cap \pi_1S$, and by our assumptions above, $G_0$ is an infinite, normal subgroup.  Define $\hull_G$ to be the convex hull of the limit set of the action of $G$ on $\partial \mathbb H^2$. Since we are assuming $G \neq G_0$,  the action of $G$ on $\mathbb H^2$ is not by isometries and does not necessarily preserve $\hull_G$. However, since $G_0$  is a normal subgroup of $G$, the limit set of $G_0$ and $G$ in $\partial \mathbb H^2$  are equal and $G_0$ does act isometrically on $\mathbb H^2$ preserving $\hull_G$.
	
	Since $G$ does act by isometries on the Bass--Serre tree $T$, we can use $T$ to produce a geometric model for $G$ as follows: since $G$ is purely pseudo-Anosov and torsion free, no element of $G$ fixes any simplex of $\C(S^z)$.   Hence, $G$ acts freely on $T$ as its vertices and edges are $\Gamma$--equivariantly identified with simplices of $\C^s(S^z)$ in $\Phi^{-1}(\alpha)$. Thus, the minimal invariant subtree, $T_G$ of the action of $G$ on $T$ is again the union of axes of loxodromic element of $G$. A compact fundamental domain for this action can be found by taking the  minimal subtree containing a base vertex  $v \in T_G$ and all the translates of $v$ by a finite set of generators of $G$. Thus, the action of $G$ on $T_G$ gives $G$ a graph of groups decomposition with trivial vertex and edge groups. This proves the following lemma. 
	
	\begin{lemma} \label{L:free action of G} The group $G$ is free. Moreover, the tree $T_G$ has uniformly finite valence and a free, cocompact $G$--action. 
	\end{lemma}
	
	\begin{remark}
		Since $G_0$ is a normal, infinite subgroup of $G$, the tree $T_G$  is also the minimal invariant tree of the action of $G_0$ on $T$. Hence $T_G$ is also the union of the axes of the loxodromic elements of $G_0$.
	\end{remark}

	Since every  element of $G_0$ is loxodromic on $T_G$, and since $G$ and $G_0$ have equal limit sets in $\partial T$ and $\partial \mathbb H^2$, a similar argument as Item \eqref{I:T_hull=H-hull} for Lemma \ref{L:options_for_T_u} shows that the same conclusion holds for $T_G$ and $\hull_G$.
	
	\begin{lemma}
		A vertex $t \in T$ is a vertex of $T_G$ if and only if $\widetilde{Y}_t^\circ \cap \hull_G \neq \emptyset$.
	\end{lemma}

	\subsection{The $G_0$--quotient and its spine} \label{S:G0 quotient} Since $G_0$ acts freely on $T_G$, there is a $G_0$--equivariant embedding $T_G \to \hull_G$ sending vertices inside the component they are dual to (in a $G$--equivariant way) and sending edges to geodesic segments.  Therefore, we get a surface with a spine
	\[ T_G/G_0 = \sigma_0 \subset \Sigma_0 = \hull_G/G_0.\]
 Figure~\ref{F:Sigma_0 and its spine} gives an  example of $\Sigma_0$ and its spine $\sigma_0$.
	
	\begin{center}
		\begin{figure}[htb]
			\begin{tikzpicture}
				\node at (0,0) {\includegraphics[width=15cm]{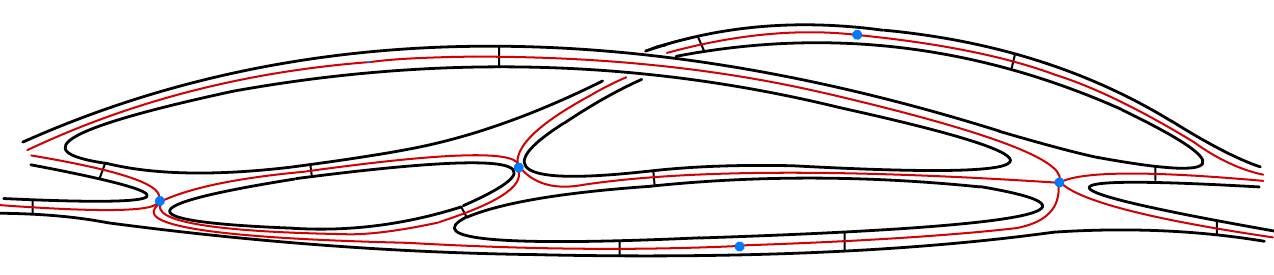}};
			\end{tikzpicture}
			\caption{Part of $\Sigma_0$ and its spine $\sigma_0 \subset \Sigma_0$.  Each edge $\ee \subset \sigma_0$ transversely intersects its dual arc $a_\ee$.} \label{F:Sigma_0 and its spine}
		\end{figure}
	\end{center}

Each edge $e \subset T_G$ intersects exactly one component $\widetilde \alpha_e \subset p^{-1}(\alpha)$ and we define
	\[ \widetilde a_e = \widetilde \alpha_e \cap \hull_G.\] 
We write $a_\ee \subset \Sigma_0$ for the image of $\widetilde a_e$ in $\Sigma_0$ where $e \subset T_G$ is an edge that projects to $\ee$; note that for any two edges $\ee,\ee'$ of $\sigma_0$, $a_
	\ee \cap \ee'$ is empty if $\ee \neq \ee'$, while $a_
	\ee \cap \ee'$ is a single point if $\ee = \ee'$.

	\subsection{Polygons and $G$--quotient} \label{S:polygons and G quotient}

	For each vertex $t \in T_G$, the intersection $\widetilde Y_t \cap \hull_G$ is an even-sided polygon with sides alternating between arcs contained in $p^{-1}(\alpha)$ and those in $\hull_G$.  Indeed, the sides in $p^{-1}(\alpha)$ are precisely the arcs $\widetilde a_e$ where $e$ is an edge of $T_G$ adjacent to $t$.  We let $\widetilde Z_t \subset \hull_G$ be this polygon corresponding to the vertex $t \in T_G$, and we write $\partial_\alpha \widetilde Z_t$ to denote the union of the sides $\widetilde a_e$ over all edges $e$ adjacent to $t$; see Figure~\ref{F:Z and Z}.

	\begin{center}
		\begin{figure}[htb]
			\begin{tikzpicture}[scale=1.25]
				\node at (0,0) {\includegraphics[width=15cm]{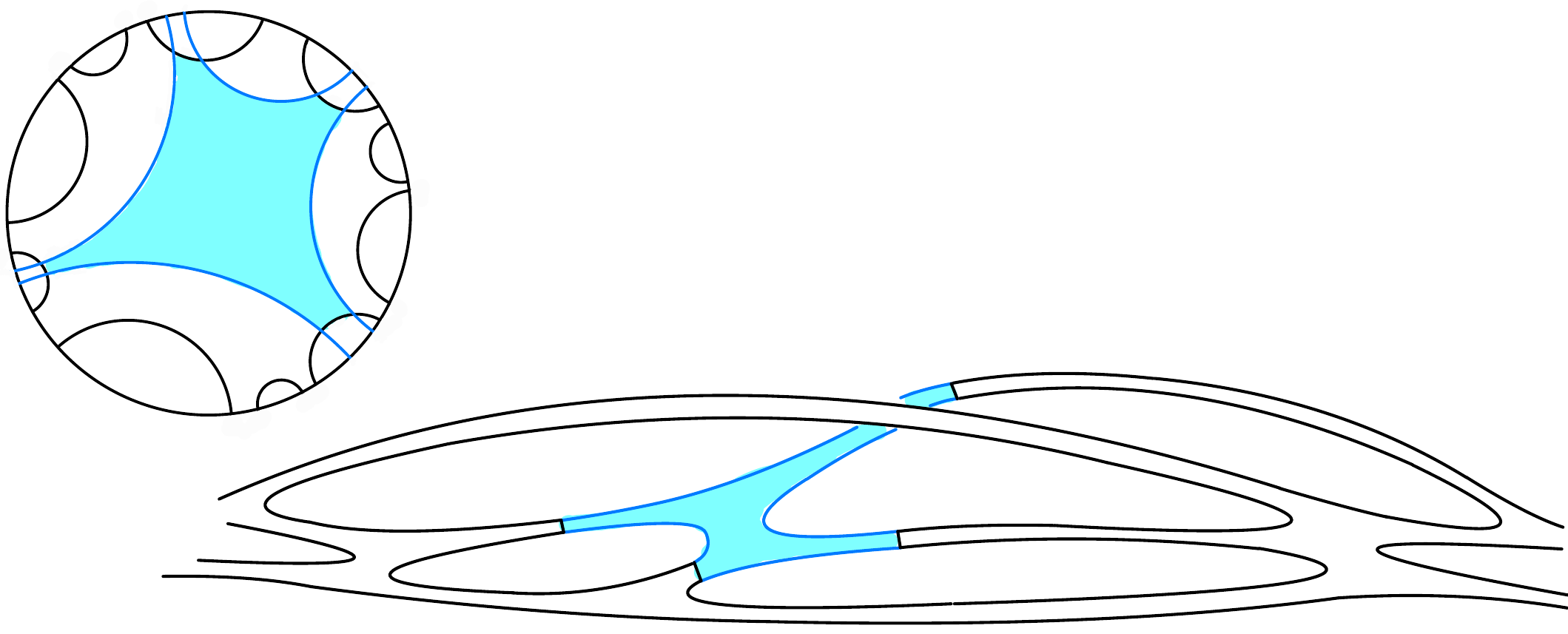}};
				\node at (-4.3,1) {$\widetilde Z_t$};
				\node at (-5.1,1.4) {$\widetilde Y_t$};
				\node at (4,-.4) {$\Sigma_0$};
				\node at (-.3,-1.7) {$Z_\tau$};
				\node at (-1.4,.6) {$p_0$};
				\draw[-\tips] (-2.8,.5) .. controls (-2.3,.5) and (-1.5,.5) .. (-.4,-.2);
			\end{tikzpicture}
			\caption{Left: The polygon $\widetilde Z_t \subset \hull_G \cap \widetilde Y_t$ (shaded).  Right: The ``image'' polygon $Z_\tau$ in $\Sigma_0$.} \label{F:Z and Z}
		\end{figure}
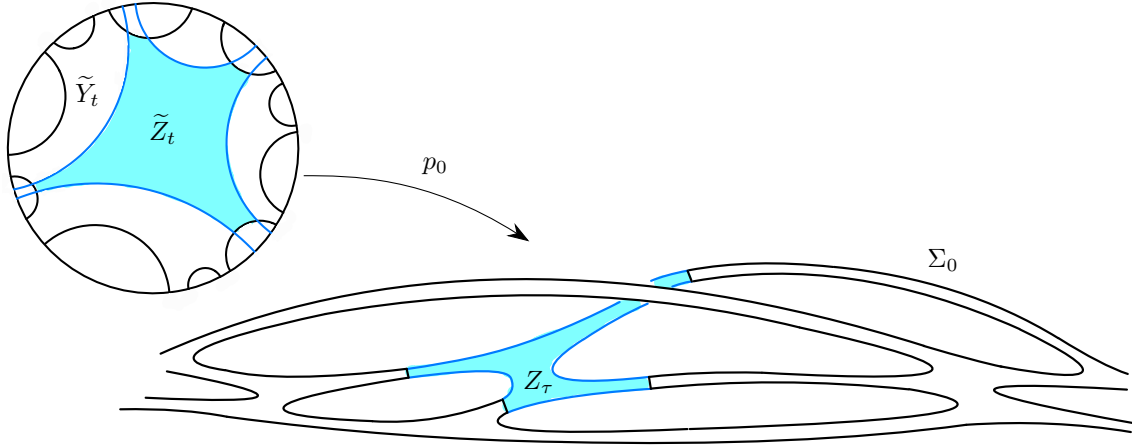
	\end{center}

	Let $\widetilde p_0 \colon \mathbb H^2 \to \widetilde S_0 = \mathbb H^2/G_0$ be the quotient by $G_0$, which contains $\Sigma_0$ as its convex core (by definition), and write $p_0 = \widetilde p_0|_{\hull_G} \colon \hull_G \to \Sigma_0$ for the restriction.  Let $\eta \colon \widetilde S_0 \to S$ be the associated covering corresponding to $G_0 < \pi_1S$, so that $\eta \circ \widetilde p_0 = p$.
	
	Now $\eta^{-1}(\alpha) \cap \Sigma_0$ is a union of the geodesic arcs $a_\ee$ over all edges $\ee$ of $\sigma_0$.  The further restriction of $p_0$ to $\widetilde Z_t$ is injective on $\widetilde Z_t \ssm  \partial_\alpha \widetilde Z_t$ and maps $\partial_\alpha \widetilde Z_t$ into $\eta^{-1}(\alpha)$.  For a vertex $\tau \in \sigma_0$, write $Z_\tau = \widetilde Z_t$ where $t$ is a vertex of $T_G$ with $p_0(t) = \tau$, and write $Z_\tau \to \Sigma_0$ to denote the restriction of $p_0$.  This map is injective, except possibly on the points of $\partial_\alpha \widetilde Z_t$.  As an abuse of notation, we write $Z_\tau \subset \Sigma_0$ (even though it is not necessarily embedded).  See Figure~\ref{F:Z and Z}.

	Since $G_0$ is a normal subgroup of $G$, we have an action of $G/G_0 \cong \mathbb Z $ on $\widetilde S_0$, and we observe that each element of $G/G_0$ acts as a lift of a power of $f$ to the covering space $\widetilde S_0$. The action of $G/G_0$ on $\widetilde S_0$ is free because the action of $G/G_0$ on $\sigma_0$ is free.
	The action of $G/G_0 \cong \mathbb Z $ does not preserve $\Sigma_0$, but we can find a homeomorphism $\frak f \colon \widetilde S_0 \to \widetilde S_0$ so that $\frak{f}(\Sigma_0) = \Sigma_0$ and $\frak f$ is properly isotopic the generator of $G/G_0$ by an isotopy that preserves $\eta^{-1}(\alpha)$. 
	If the generator sends a vertex $\tau \in \sigma_0$ to a vertex $\tau' \in \sigma_0$, then $\frak f(Z_\tau) = Z_{\tau'}$; indeed, we use this to define the isotopy of the generator to the map $\frak f$.
	By further proper isotopy preserving $\eta^{-1}(\alpha)$ and $\Sigma_0$, we may assume $\frak f(\sigma_0) = \sigma_0$.  The action of $\langle \frak f \rangle$ on $\Sigma_0$ is a topological covering space action with compact quotient $\Sigma$ containing a spine $\sigma = \sigma_0/\langle \frak f \rangle$.  
	
	We note that the map $g \mapsto \frak f^{\phi(g)}$ defines a homomorphism $G \to \langle \frak f \rangle$ that descends to an isomorphism $G/G_0 \cong \langle \frak f \rangle$.  Since $\phi|_G$ is surjective, we also note that $\frak f$ is isotopic to a lift of $f$.  Moreover, the projection $T_G \to \sigma_0$ is equivariant with respect to this homomorphism.

\section{Reduction to a diameter bound for $p_0(T_G \cap T_u)$}

In this section we reduce the proof of Theorem \ref{T:reducible} to proving that the diameter of $p_0(T_G \cap T_u)$ is uniformly bounded for all simplices $u \subset \C(S^z)$. This has two steps. First we show that Theorem \ref{T:reducible} follows from a uniform bound on the diameter of $T_G \cap T_u$. Second we show that the diameter bound on $T_G \cap T_u$ follows from the a priori weaker bound on the diameter of $p_0(T_G \cap T_u)$. The proof of Theorem \ref{T:reducible} will then be completed in \S \ref{S:Bounding_intersection_in_sigma_0} where we verify that   $p_0(T_G \cap T_u)$ is uniformly bounded.

\subsection{First reduction}
We give the proof of Theorem \ref{T:reducible} assuming the following proposition.

\begin{proposition} \label{P:bounded diameter} Given $G < \Gamma$ finitely generated and purely pseudo-Anosov in $\Mod(S^z)$, there exists $D >0$ so that for all $u \subset \C(S^z)$ we have
	\[ \diam(T_G \cap T_u) \leq D.\]
\end{proposition}

\begin{proof}[Proof of Theorem~\ref{T:reducible} assuming Proposition~\ref{P:bounded diameter}] Let $P \colon T \to T_G$ be the closest point projection.  Observe that $P$ maps any connected subset of $T \ssm T_G$ to a point.  In particular, for any geodesic segment $\sigma$ outside $T_G$, $P(\sigma)$ is a point.  Now suppose $u \subset \C(S^z)$ is any simplex and $\sigma$ is a geodesic in $T_u$.  Then $\sigma = \sigma_0 \sigma_1\sigma_2$, where $\sigma_1$ is a (possibly empty) geodesic segment in $T_G \cap T_u$ and $\sigma_0,\sigma_2$ meet $T_G$ in at most one point.  It follows that $P(\sigma)$ is either a point, or $P(\sigma) = \sigma_1$.  In either case, $\diam(P(\sigma)) \leq D$ by Proposition~\ref{P:bounded diameter}.
	
	Fix a vertex $t \in T_G$ and let $u \in  \C(S^z)$ be a curve in the simplex that is the image of $t$ in $\Phi^{-1}(\alpha) \subset \C(S^z)$. Consider the orbit map $G \to \C(S^z)$ given by $g \to g (u)$.  Write $d_T$ for the (geodesic) metric on $T_G$ and $d_\C$ for the metric on the $1$--skeleton of  $\C(S^z)$. 
	\begin{claim}\label{CL:lower_bound} $d_T( t,g(t) ) \leq 2Dd_\C( u,g(u) ) + D$.
	\end{claim}
	Assuming the claim, we complete the proof of the theorem.  Fix a finite generating set for $G$ and write $d_G$ for the word metric.  A standard application of the triangle inequality implies that the orbit map $G \to G \cdot u \subset \C(S^z)$ is lipschitz with respect to $d_G$.  Next, note that the orbit map $G \to G \cdot t \subset T_G$ is a $(\kappa,\lambda)$--quasi-isometry, for some $\kappa,\lambda$, and thus by the claim
	\[ d_G(1,g) \leq \kappa d_T(t,g (t)) + \lambda \leq 2\kappa D d_\C(u,g (u))  + \kappa D + \lambda.\]
	Therefore, the orbit map $G \to G \cdot u \subset \C(S^z)$ is a quasi-isometric embedding, and hence $G$ is convex cocompact by Theorem~\ref{T:ccc char}.
	
	\begin{proof}[Proof of Claim \ref{CL:lower_bound}]
		Let $n = d_\C(u,g (u))$ and write $u = u_0,u_1,\ldots,u_n = g (u)$ for the vertices of a $\C(S^z)$--geodesic from $u $ to $g(u)$.  Consider the set of simplices
		\[ w_{2j} = \{u_j\} \mbox{ and } w_{2i+1} = \{u_i,u_{i+1}\},\]
		for $j=0,\ldots,n$ and $i=0,\ldots,n-1$.  In particular, $w_{2j,2j+2} \subset w_{2j+1}$ for all $j=0,\ldots,n-1$.  By Lemma \ref{L:nested_simplices}, this implies 	$$ T_{w_k} \cap T_{w_{k+1}} \neq \emptyset 
		\text{ for all } k = 0,\ldots, 2n-1.$$ We also observe that since $u$ is a vertex of the simplex defined by $t$, we have $\{t\}= T_t \subseteq T_u$  and likewise  $\{g(t)\} = T_{g(t)} \subseteq T_{g (u)}$.
		
		Now construct a path $\gamma \colon [0,2n+1] \to T$ by
		\begin{itemize}
			\item $\gamma(0) = t$,
			\item $\gamma(2n+1) = g (t)$,
			\item $\gamma([k,k+1]) \subset T_{w_k}$ is a geodesic segment.
		\end{itemize}
		This is possible because $T_{w_k} \cap T_{w_{k+1}} \neq \emptyset$ for all $k=0,\ldots,2n-1$. Hence, we can define $\gamma(k+1)$ to be any point in the intersection of these subtrees, and then take $\gamma\vert_{[k,k+1]}$ to be a geodesic segment in $T_{w_k}$ connecting the points $\gamma(k),\gamma(k+1) \in T_{w_k}$.  At the endpoints, we note that $\gamma(0) = t \in T_u = T_{w_0}$ and $\gamma(2n+1) = g (t) \in T_{g (u)} = T_{w_{2n}}$.
		
		Now consider the path $P \circ \gamma \colon [0,2n+1] \to T_G$.  As noted above, since
		\[ \gamma([k,k+1]) \subset T_{w_k},\]
		Proposition \ref{P:bounded diameter} implies $\diam(P \circ \gamma([k,k+1])) \leq D$.  Now, $P \circ \gamma$ is a path in $T_G$ between $t$ and $g (t)$, and thus we have
		\[ d_T(t,g (t)) \leq \diam(P \circ \gamma) \leq (2n+1)D = 2Dd_\C(u,g (u)) + D,\]
		which  proves the claim.
	\end{proof}
	Having proved Claim~\ref{CL:lower_bound}, we have  proved the theorem assuming Proposition~\ref{P:bounded diameter}.
\end{proof}

\subsection{Second reduction} \label{S:second reduction}

Having reduced the proof of Theorem~\ref{T:reducible} to Proposition~\ref{P:bounded diameter}, which asserts a uniform bound on the diameter of $T_G \cap T_u$, we proceed to our second reduction.  The goal of this section is thus to deduce such a uniform bound from the following bound in $\sigma_0 = T_G/G_0$.

\newcommand{\Pweakbound}{Given $G < \Gamma$ finitely generated and purely pseudo-Anosov in $\Mod(S^z)$, there exists $D' > 0$ so that for any simplex $u \subset \C(S^z)$,
	\[ \diam(p_0(T_G \cap T_u)) \leq D',\]
	where the diameter of $p_0(T_G \cap T_u)$ is computed in $\sigma_0$.}

\begin{proposition} \label{P:weak bounded diameter} \Pweakbound \end{proposition}

We postpone the proof of Proposition~\ref{P:weak bounded diameter} to \S \ref{S:Bounding_intersection_in_sigma_0} and focus this subsection on using Proposition \ref{P:weak bounded diameter} to prove Proposition~\ref{P:bounded diameter}. Our proof of Proposition~\ref{P:bounded diameter} has two parts. First we show that it suffices to verify the proposition for multicurves $u$ where $p_0(T_G \cap T_u)$ lands sufficiently deep in a specific subgraph of $\sigma_0$. Then, we verify that $p_0(T_G \cap T_u)$ is uniformly bounded for these ``deep'' multicurves.  

The following easy fact allows us to adjust simplices by elements of $G$.

\begin{lemma} \label{L:translating suffices} For any $g \in G$ and simplex $u \subset \C(S^z)$ we have
	\[ g(T_G \cap T_u) = T_G \cap T_{g (u)}.\]
\end{lemma}
\begin{proof} Since $g \in G$, we have $g (T_G) = T_G$.  Since $K_{g (u)} = g K_u g^{-1}$, by considering the two cases (minimal invariant subtree or maximal fixed subtree), we see that $g (T_u) = T_{g (u)}$.  Therefore, we have
	\[ T_G \cap T_{g (u)} = g (T_G) \cap g (T_u) = g(T_G \cap T_u).\qedhere\]
\end{proof}

Lemma \ref{L:options_for_T_u}\eqref{I:fixed_sets} says that if $T_u$ has finite diameter, then in fact the diameter is at most $1$. Thus it will suffice to examine $T_G \cap T_u$ only for simplices $u \subset \C(S^z)$ where $T_u$ has infinite diameter. In particular, by Lemma~\ref{L:options_for_T_u}\eqref{I:non-surviving} we can assume all curves of $u$  are surviving.

\subsubsection{Reduction to deep simplices}

Let $e_1,\ldots,e_r$ be a set of $\langle \frak f \rangle$--orbit representatives of edges of $\sigma_0$.  Let $\sigma_1 \subset \sigma_0$ be a connected subgraph containing $e_1,\ldots,e_r$ so that the distance in $\sigma_0$ from any edge $e_i$ to a point outside $\sigma_1$ is at least $D'+2$, where $D'$ is the constant from Proposition \ref{P:weak bounded diameter}.
We further assume the following for each boundary component $\delta$ of $\partial \Sigma_0$:  if $\delta^*$ is the minimal length loop in $\sigma_0$ that is freely homotopic in $\Sigma_0$ to $\delta$, then there exists $i \in \mathbb Z$ so that $\frak f^i(\delta^*) \subset \sigma_1$.  This is possible since there are only finitely many $\langle \frak f \rangle$--orbits of boundary components of $\Sigma_0$. Since $\sigma_0$ has no valence $1$ vertices by virtue of being the quotient of axes of loxodromics, we can enlarge $\sigma_1$ to ensure it also has no valence $1$ vertices and that $\sigma_1$ contains all edges with endpoints in its vertex set. We say that a simplex $u \subset \C^s(S^z)$ is \emph{deep} if $N_2(p_0(T_G \cap T_u)) \subset\sigma_1$ where $N_2(\cdot)$ is the $2$--neighborhood in $\sigma_0$. 

The next lemma, combined with Lemma~\ref{L:translating suffices}, shows that it suffices to verify Proposition~\ref{P:bounded diameter} for deep simplices.

\begin{lemma} \label{L:translating deep} For any simplex $u \subset \C(S^z)$, there exists $g \in G$ so that $g(u)$ is a deep simplex.  That is,
	\[ N_2(p_0(T_G \cap T_{g(u)})) \subset \sigma_1.\]
\end{lemma}

\begin{proof} For any $u$, there is $j \in \mathbb Z$ and one of the chosen $\langle \frak f \rangle$--orbit representatives of edges $e_i$ so that
	\[ e_i \subset \frak f^j(p_0(T_G \cap T_u)). \]
	Then $N_2(\frak f^j(p_0(T_G \cap T_u))) \subset \sigma_1$ by Proposition~\ref{P:weak bounded diameter}.
	
	Now we let $g \in G$ be any element that maps to $\frak f^j$ by the homomorphism $G \to \langle \frak f \rangle$.  Since $p_0|_{T_G} \colon T_G \to \sigma_0$ is equivariant with respect to this homomorphism, Lemma~\ref{L:translating suffices} implies
	\[ p_0(T_G \cap T_{g(u)}) = p_0(g (T_G \cap T_u)) = \frak f^j(p_0(T_G \cap T_u)).\]
	Combining this with the previous paragraph proves the lemma.
\end{proof}

\subsubsection{Subtree decomposition and bounding $T_G \cap T_u$ for deep simplices} \label{S:subtrees}
We now use Proposition \ref{P:weak bounded diameter} to uniformly bound the diameter of $T_G \cap T_u$ when $u$ is a deep simplex (and thus for any simplex, by Lemmas~\ref{L:translating suffices} and \ref{L:translating deep}). We start by dividing the vertices of $T_G \cap T_u$ into two sets.

\begin{definition}
	Given a simplex $u \subset \C^s(S^z)$, say that a vertex $t \in T_G \cap T_u$ is of {\em hull type} if
	\[ \hull_G \cap \hull_u \cap  \widetilde Y_t^\circ \neq \emptyset.\]
	Any vertex that is not hull type is called {\em parallel type}.
\end{definition}
The reason for the name ``parallel type" comes from Lemma \ref{L:Parallel_type_vertices} below, which says parallel type vertices must arise from single components of $\partial \hull_G$ and $\partial \hull_u$ running parallel to each other.

The next lemma verifies that the set of hull type vertices span a subtree of $T_G \cap T_u$.

\begin{lemma} \label{L:hull type subtree} If the set of hull type vertices is nonempty, then it spans a subtree of $T_G \cap T_u$.  That is, every vertex of the smallest subtree containing all the hull type vertices is of hull type.
\end{lemma}
\begin{proof} If $t,s \in T_G \cap T_u$ are hull type vertices, let $x,y \in \hull_G \cap \hull_u$ be points with $x \in \widetilde Y_t^\circ$ and $y\in \widetilde Y_s^\circ$.  Then the geodesic $[x,y] \subset \mathbb H^2$ is contained in $\hull_G \cap \hull_u$ by convexity.  Adjusting our equivariant map $\mathbb H^2 \to T$ if necessary (see \S\ref{S:actions on H and T}), we may assume it sends $[x,y]$ to a geodesic from $t$ to $s$ in $T_G \cap T_u$.  Every vertex of this geodesic is therefore of hull type. 
\end{proof}

We call the subtree of $T_G \cap T_u$ from Lemma \ref{L:hull type subtree} the {\em hull subtree}, and denote it $T^\hull_{u,G}$.  Each maximal connected subgraph of the complement of $T^\hull_{u,G}$ is also a subtree of $T_G \cap T_u$. We call these components the {\em parallel  subtrees} of $T_G \cap T_u$.  To avoid arguing in separate cases, we allow the possibility that $T^\hull_{u,G}$ is empty (i.e.~if there are no hull type vertices) in which case $T_G \cap T_u$ is the unique parallel subtree.  If $T_G \cap T_u = T^\hull_{u,G}$, then we consider any parallel subtree to be empty.

Before bounding the diameter of the hull and  parallel subtrees, we need some additional terminology.   Let $\Sigma_1$ be the compact subsurface of $\Sigma_0$ defined by
\[\Sigma_1 = \bigcup_{\tau \in \sigma_1^{(0)}} Z_\tau. \] 
We note that if $\tau,\tau'$ are endpoints of an edge $\ee \subset \sigma_1$, then there are corresponding arcs $\partial_\alpha Z_\tau,\partial_\alpha Z_{\tau'}$ which are identified in $\Sigma_0$ (hence in $\Sigma_1$) and which transversely intersect $\ee$. Conversely, if $\tau,\tau' \in \sigma_1^{(0)}$ are vertices for which arcs of $\partial_\alpha Z_\tau$ and $\partial_\alpha Z_{\tau'}$ are identified in $\Sigma_1$, then this arc is transverse to an edge $\varepsilon \subset \sigma_0$, which must be in $\sigma_1$ since its endpoints are.  It follows that the inclusion $\sigma_1 \to \Sigma_1$ is a homotopy equivalence. 
Let $G_1 < G_0$ be the image of the fundamental group of $\Sigma_1$ in $G_0 = \pi_1\Sigma_0$. Equivalently, $G_1 < G_0$ is the image of the fundamental group of $\sigma_1$ inside $G_0 = \pi_1 \sigma_0$.

Let $\widetilde \sigma_1 \subset T_G$ be the component of $p_0^{-1}(\sigma_1)$ that is $G_1$--invariant and define 
\[ \hull_G^1  =  \bigcup_{t \in \widetilde \sigma_1^{(0)}} \widetilde Z_t.\]
Note that $\hull_G^1$ is the minimal, closed, $G_1$--invariant subspace of $\hull_G$ that projects to $\Sigma_1$.
We also let $\hull_{G_1}$ be the convex hull of the limit set of $G_1$.  

Let $R$ be the maximum of the diameters of the polygons $Z_\tau$ over all vertices $\tau \in \sigma_1$ and observe that
\[ \hull_G^1 \subset N_R(\hull_{G_1}),\]
since $\sigma_1$ contains no valence $1$ vertices.  To see this, note that any closed loop in $\sigma_1$ without backtracking that visits every vertex of $\sigma_1$, has geodesic representative $\gamma$ in $\Sigma_0$ that meets $Z_\tau$ for every vertex $\tau \in \sigma_1$.  Therefore, for every vertex $t \in \widetilde \sigma_1$, there is a geodesic in the preimage of $\gamma$ that is invariant by an infinite cyclic subgroup of $G_1$ and passes through $\widetilde Z_t$.   Since any such geodesic is contained in $\hull_{G_1}$, every point of $\hull_G^1$ is within $R$ of a point of $\hull_{G_1}$.

We now explain how to bound the diameter of the hull subtree of $T_G \cap T_u$ for deep simplices.

\begin{lemma} \label{L:bounding the hull subtree} There is a constant $D_\hull>0$ so that for any deep simplex $u \subset \C^s(S^z)$, the diameter of $T^\hull_{u,G}$ is at most $D_\hull$.
\end{lemma}

\begin{proof} We may assume that $T_u$ has infinite diameter, since otherwise it has diameter at most $1$, according to Lemma~\ref{L:options_for_T_u}\eqref{I:fixed_sets}, and the conclusion is trivial.  In particular, $(\widetilde Z_t \ssm \partial_\alpha \widetilde Z_t) \cap \hull_u \neq \emptyset$ if and only if $t \in T_{u,G}^\hull$ by Lemma~\ref{L:options_for_T_u}\eqref{I:T_hull=H-hull}.
	Since $\hull_G = \bigcup_{t \in T_G^{(0)}} \widetilde Z_t$ and $T_{u,G}^\hull \subset \widetilde \sigma_1$  (because $u$ is a deep simplex) we have
	\[ \hull_G \cap \hull_u = \hull_G^1 \cap \hull_u.\]
	
	Choose our equivariant map $\mathbb H^2 \to T$ to be Lipschitz (see \S\ref{S:actions on H and T}).   This map sends $\hull_G \cap \hull_u$ to a set of Hausdorff distance at most $\frac12$ from  $T_{u,G}^\hull$, and thus it suffices to prove a bound on the diameter of the intersection $\hull_G \cap \hull_u$.  
	
	To prove such a bound, first observe that
	\[ \hull_G \cap \hull_u = \hull_G^1 \cap \hull_u \subset N_R(\hull_{G_1}) \cap \hull_u.\]
	Now $G_1$ is finitely generated because $\Sigma_1$ is compact and $G_1$ is a purely pseudo-Anosov subgroup of $\pi_1S < \Mod(S^z,z)$ because $G_0$ is a purely pseudo-Anosov subgroup.   The argument in \S5 of \cite{kentleiningerschleimer} shows that if $H <\pi_1S < \Mod(S^z,z)$ is finitely generated and purely pseudo-Anosov, then  there is uniform bound on the diameter of $N_R(\hull_H) \cap \hull_u$. In particular, there is a bound on the diameter of $N_R(\hull_{G_1}) \cap \hull_u$ determined only by $G_1$, which thus also bounds the diameter of $\hull_G \cap \hull_u$.
\end{proof}
\begin{remark} The proof in \S5 of \cite{kentleiningerschleimer} actually proves a bound on $N_1(\hull_H) \cap \hull_u$, but the ``1" was an arbitrary choice, and the same proof applies replacing $1$ with any constant $R >0$.
\end{remark}

To bound the diameter of the parallel subtrees,  we need the following result, which  justifies the name of ``parallel type'' for the vertices that are not hull type.

\begin{lemma}\label{L:Parallel_type_vertices}
	Let $u \subset \C^s(S^z)$ be a multicurve such that $T_u$ has infinite diameter, and let  $t_0,\dots,t_n$ be the vertices of  an edge path  in $T_G \cap T_u$. Let $e_i$ be the edge from $t_{i-1}$ to $t_i$ and $\widetilde \alpha_i$ be the geodesic in $p^{-1}(\alpha)$ that is dual to the edge $e_i$. If each $t_i$ is of parallel type, there then exists geodesics  $\delta_G \subseteq \partial \hull_G$ and $\delta_u \subseteq \partial \hull_u$ so that 
	\begin{itemize}
		\item $\delta_u$ and $\delta_G$ intersect each $\widetilde \alpha_i$ transversely;
		\item $\delta_u$ and $\delta_G$ do not intersect in $\widetilde{Y}_{t_i}$ for any $i \in \{0,\dots,n\}$
	\end{itemize}
\end{lemma}

\begin{proof}
	Note, each $t_i$ being of parallel type means that any geodesics satisfying the first item must automatically satisfy the second. Hence it suffices to produce the geodesics  $\delta_G \subseteq \partial \hull_G$ and $\delta_u \subseteq \partial \hull_u$ that intersect each $\widetilde \alpha_i$.
	
	Convexity ensures  that $\hull_G$ and $ \hull_u$ intersect each $\widetilde \alpha_i$ in  a (possibly non-compact) non-empty, closed interval.  If the vertices $t_{i-1}$ and $t_i$ are both of parallel type, then these intervals are disjoint, hence there must be $x_i \in \partial \hull_G \cap \widetilde \alpha_i$ and $y_i \in \partial \hull_u \cap \widetilde \alpha_i$ so that the open interval of $\widetilde \alpha_i$ between $ x_i$ and $y_i$ does not intersect either $\hull_G$ or $\hull_u$. Moreover, the $x_i$ and $y_i$ must be arranged so that the geodesic from $x_i$ to $x_{i+1}$ does not cross the geodesic from $y_i$ to $y_{i+1}$. Let $\widetilde \alpha_i^+ \in \partial \mathbb H^2$ be the endpoint of the subray of $\widetilde \alpha_i$ starting at $y_i$ and passing through $x_i$. Similarly, let $\widetilde \alpha_i^-\in \partial \mathbb H^2$ be the endpoint of the subray of $\widetilde \alpha_i$ starting at $x_i$ and passing through $y_i$; see Figure \ref{F:parallel_type}. Since the geodesic from $x_i$ to $x_{i+1}$ does not cross the geodesic from $y_i$ to $y_{i+1}$, there are disjoint arcs $I_+,I_- \subset \partial \mathbb H^2$ so that $\widetilde \alpha_1^+, \dots \widetilde \alpha_n^+ \subset I_+$ and  $\widetilde \alpha_1^-, \dots \widetilde \alpha_n^- \subset I_-$.
	
	Let $\delta_i$ be the component of $\partial \hull_u$ that contains $y_i$ for $i \in \{0,\dots, n-1\}$. If $\delta_i$ does not also include $y_{i+1}$, then $\delta_i$ must have an endpoint on the arc of $\partial \mathbb{H}^2$ between $\widetilde \alpha^+_i$ and $\widetilde \alpha^+_{i+1}$ (contained in $I_+$). But that would require $\delta_i$ to cross the geodesic from $x_i$ to $x_{i+1}$ as shown in Figure \ref{F:parallel_type}. 
	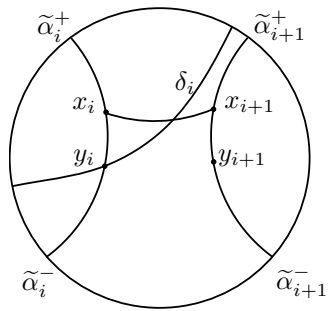
\begin{figure}[htb]
		\begin{tikzpicture}
			\node at (0,0) {\includegraphics[width=4cm]{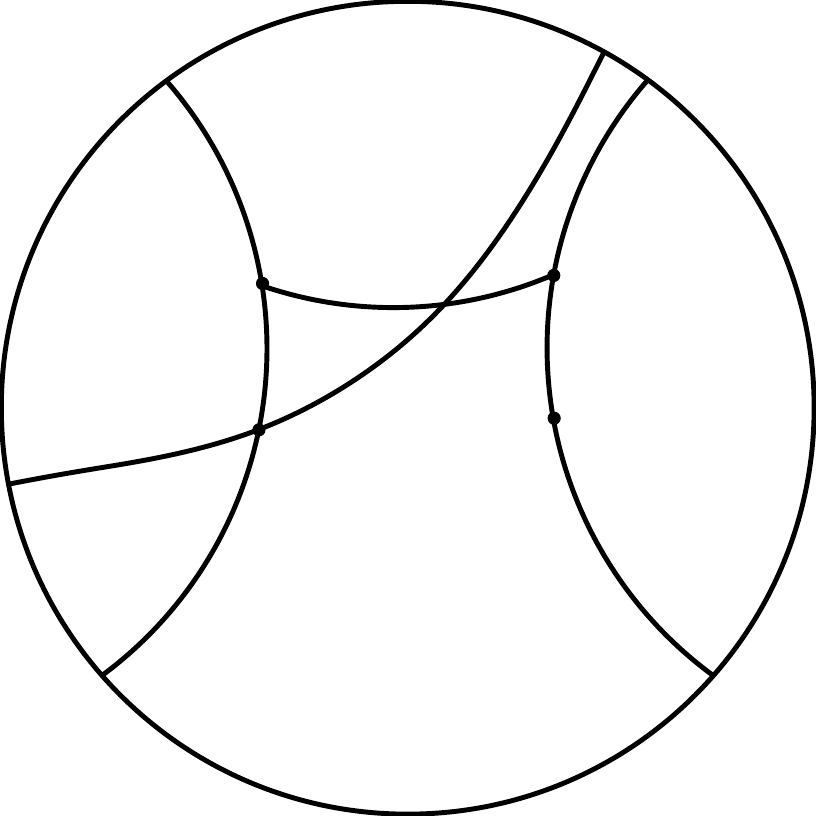}};
			\node at (-1,0) {$y_i$};
			\node at (1.1,0) {$y_{i+1}$};
			\node at (-1,.7) {$x_i$};
			\node at (1.2,.7) {$x_{i+1}$};
			\node at (-1.4,1.75) {$\widetilde \alpha_i^+$};
			\node at (1.6,1.75) 	{$\widetilde \alpha_{i+1}^+$};
			\node at (-1.6,-1.65) {$\widetilde \alpha_i^-$};
			\node at (1.9,-1.65) 	{$\widetilde \alpha_{i+1}^-$};
			\node at (.34,1) { $ \delta_i$};	
		\end{tikzpicture}
		\caption{Arrangement of $x_i$ and $y_i$.} \label{F:parallel_type}
	\end{figure}
	\noindent Since the geodesic from $x_i$ to $x_{i+1}$ is contained in $\hull_G$ and $\delta_i$ is contained in $\hull_u$, this would contradict that $t_{i+1}$ is a parallel type vertex for $i \in \{0,\dots, n-1\}$. Hence there is a single boundary component $\delta_u \subset \partial \hull_u$ that contains all the $y_i$. A completely analogous argument shows that there is a single boundary component $\delta_G \subseteq \partial \hull_G$ that contains all the $x_i$.
\end{proof}

We also need this basic fact about quadrilaterals in the hyperbolic plane, the proof of which is left as an exercise in hyperbolic geometry.

\begin{lemma}\label{L:wide_rectangles_are_narrow}
	For each $r \geq0$ and $0<\epsilon \leq 1$, there exists $C \geq 0$  so the following holds. Let $\gamma_1,\gamma_2,\gamma_3, \gamma_4$ be the 4-sides of a convex quadrilateral in $\mathbb{H}^2$, labeled so that $\gamma_1$ is opposite $\gamma_3$. If $d(\gamma_1,\gamma_3) \geq C$, then there exists subsegments  $s_2\subseteq \gamma_2$, $s_4 \subseteq \gamma_4$,  each of length at least $r$, so that $s_2 \subseteq N_\epsilon(s_4)$ and $s_4 \subseteq N_\epsilon(s_2)$.
\end{lemma}

We now bound the diameter of the parallel subtrees.

\begin{lemma} \label{L:bounding parallel subtrees} There is a constant $D_{\parallel}>0$ so that for any deep simplex $u \subset \C^s(S^z)$, the diameter of any parallel subtree of $T_G \cap T_u$ is at most $D_{\parallel}$.
\end{lemma}

\begin{proof}
	By the Collar Lemma \cite{Ke}, we  may (and will) assume that the hyperbolic metric on $S$ is chosen so that each component of $\alpha$ is short enough to ensure that the distance between two different geodesics in $p^{-1}(\alpha)$ is at least 1.  Let $t_0,\dots,t_n$ be the vertices of a geodesic edge path in one of the parallel subtrees of $T_G \cap T_u$, then let $\widetilde{\alpha}_i$ be the geodesic of $p^{-1}(\alpha)$ that is dual to the edge from $ t_{i-1}$ to $t_i$. By Lemma \ref{L:Parallel_type_vertices}, there are geodesics $\delta_G \subseteq \partial \hull_G$ and $\delta_u \subseteq \partial \hull_u$ that form a convex quadrilateral  with $\widetilde \alpha_1$ and $\widetilde \alpha_n$.
	We first show that if $n$ is large enough, then $p_0$ maps $\delta_G$ onto a simple closed geodesic $c$ that is contained in $\partial \Sigma_0 \cap \partial \Sigma_1$.
	
	Since $\Sigma_1$ is the union of the polygons $Z_\tau$ for $\tau \in \sigma_1^{(0)}$, every component $c \subset \partial \Sigma_1$ is either a closed curve in $\partial \Sigma_0$ or  $c \cap \partial \Sigma_0$ is a disjoint union of geodesic arcs. Since $\Sigma_1$ is compact, there exists $L >0$ so that every component of $\partial \Sigma_0 \cap \partial \Sigma_1$ has length at most $L$.
	
	Let $\delta_G'$ be the subsegment of $\delta_G$ between $\widetilde \alpha_1$ and $\widetilde \alpha_n$. Now $\delta'_G$ is the concatenation of arcs in $\partial \widetilde{Z}_t \ssm \partial _\alpha \widetilde{Z}_t $ where $t \in \{t_0,\dots,t_{n}\}$. Since the geodesics in $p^{-1}(\alpha)$ are at least $1$ apart, if $n\geq L+2$, then the length of $\delta'_G$  is at least $L+1$. Moreover,  $p_0(t_i) \in \sigma_1$ because each $t_i \in T_G \cap T_u$ and $u$ is a deep simplex. This  means $p_0(\delta'_G \cap \partial \widetilde{Z}_{t_i})$  is contained in $\partial \Sigma_0 \cap \partial \Sigma_1$ for each $i \in \{1,\dots, n-1\}$, and thus $p_0(\delta'_G) \subset \partial \Sigma_0 \cap \partial \Sigma_1$. Since the components of $\partial \Sigma_0 \cap \partial \Sigma_1$ that are not closed curves in $\partial \Sigma_0$  are all arcs of length at most $L$,  $p_0$ maps $\delta'_G$ onto a closed curve $c \subset \partial \Sigma_0$. Because $c$ is a closed geodesic and $\delta'_G$ is a subsegment of $\delta_G$, this means the entire geodesic $\delta_G$ must also map onto $c$. 
	
	Let $c_1,\dots, c_k$ be the closed curves in $\partial \Sigma_0 \cap \partial \Sigma_1$. For each $c_i$, there is a geodesic curve   $\gamma_i\subset S$ so that the element of $G_1=\pi_1 \Sigma_1 \leq \pi_1 S$ that corresponds to $c_i$ is represented in $\pi_1S < \Mod(S^z;z)$ by the point push of $z$ along $\gamma_i^{-1}$. Because $G_1< G$ and $G$ 
	is purely pseudo-Anosov, Theorem~\ref{T:Kra} says each $\gamma_i$ fills $S$, and hence these are not simple. Moreover, $\widetilde p_0^{-1} (c_i) \subseteq p^{-1}(\gamma_i)$ where $\widetilde{p}_0$ is the covering map $\mathbb H^2 \to \widetilde{S}_0$. 
	
	The following claim puts a bound on how long a lift of a simple closed curve can travel close to a lift of one of the $\gamma_i$.
	
	\begin{claim}\label{CL:simple_cannot_follow_non-simple}
		There are $r\geq 0$ and $0 <\epsilon \leq 1$ independent of $u$ so that for any $i \in \{1,\dots,k\}$ and any geodesic $\widetilde \gamma \in p^{-1}(\gamma_i)$ the following holds. Let $\beta \subset S$ be a closed curve and $\widetilde \beta$ be a geodesic in  $p^{-1}(\beta)$. If  $N_\epsilon( \widetilde \beta) \cap \widetilde \gamma$ contains a geodesic of length at least $r$, then $\beta$ is not simple.  
	\end{claim}
	
	\begin{proof}
		Let $\epsilon \leq 1/16$ be small enough so that if $x$ is a self intersection point of one of the $\gamma_i$, the $8\epsilon$-neighborhood of $x$ on $S$ is isometric to the $8\epsilon$-ball in $\mathbb H^2$. Let $r_0$ be the maximum of all the lengths of all the $\gamma_i$, then let $r = 3r_0 +1$. These $\epsilon$ and $r$ depend on the hyperbolic metric on $S$ and the group $G$, but not on the multicurve $u$.
		
		Let $\beta \subset S$ be a closed curve, then let $\widetilde \beta \in p^{-1}(\beta)$  and $\widetilde \gamma \in p^{-1}(\gamma_i)$ be as described in the statement of the claim. Fix a self-intersection point $x$ of  $\gamma_i$.  	Since $N_\epsilon( \widetilde \beta)$ contains a subsegment of $\widetilde \gamma$ of length at least $r$ and $r$ is more than twice as long as the length of $\gamma_i$, there must exist  $\widetilde y_1, \widetilde w_1, \widetilde y_2, \widetilde w_2 \in \widetilde \beta$ so that the geodesic on $S$ that connects $p(\widetilde y_1)$ and $p(\widetilde w_1)$ and the geodesic that connects $p(\widetilde y_2)$ and $p(\widetilde w_2)$ must cross; see Figure \ref{F:non-simple}. Since these geodesics are subsegments of $\beta = p(\widetilde \beta)$, we have that $\beta$ cannot be simple. 
	\end{proof}

	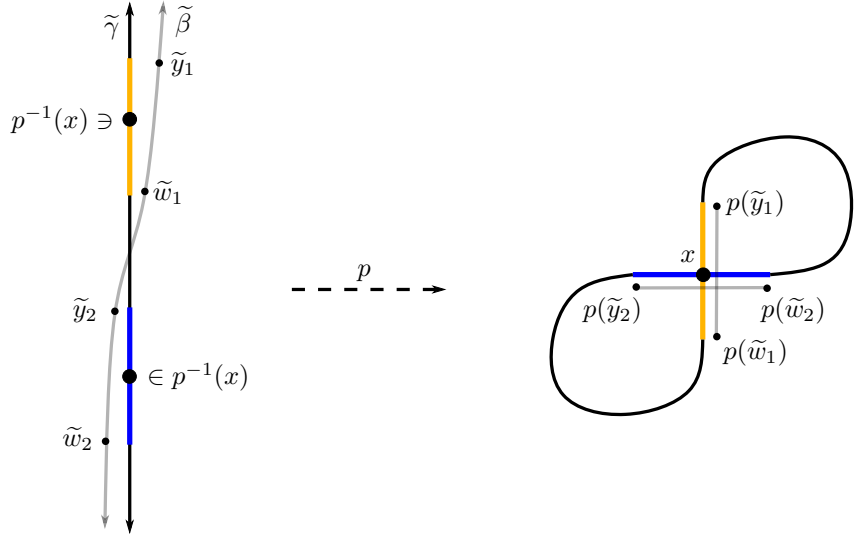
\begin{figure}[htb]
		\begin{tikzpicture}
			\node[anchor=south west,inner sep=0] (image) at (0,0) {\includegraphics[width=10cm]{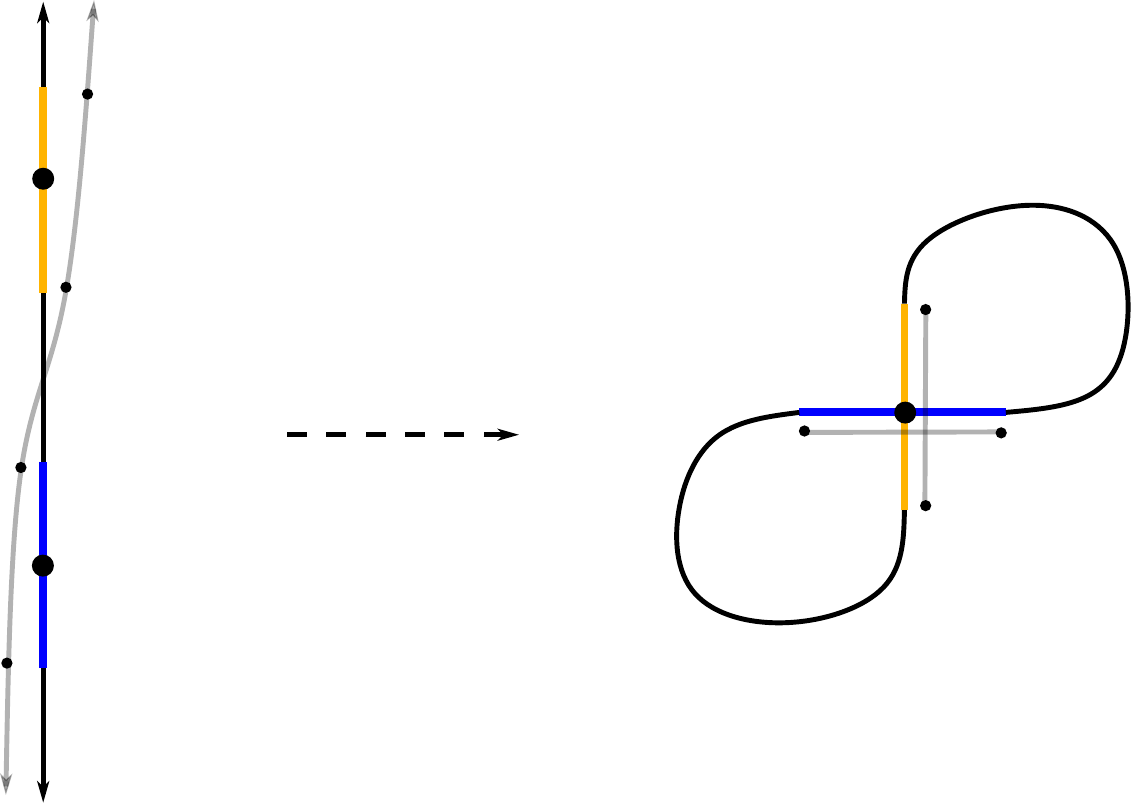}};
			\begin{scope}[x={(image.south east)},y={(image.north west)}]
				\node at (.11,.88) {$\widetilde y_1$};
				\node at (.09,.64) {$\widetilde w_1$};
				\node at (-.02,.42) {$\widetilde y_2$};
				\node at (-.03,.18) {$\widetilde w_2$};	
				\node at (.015,.95) {$\widetilde \gamma$};
				\node at (.11,.96) {$\widetilde \beta$};
				\node at (.35,.49) {$p$};
				\node at (.78,.52) {$x$};
				\node at (.87,.62) {$p(\widetilde y_1)$};
				\node at (.87,.34) {$p(\widetilde w_1)$};
				\node at (.68,.42) {$p(\widetilde y_2)$};
				\node at (.92,.42) {$p(\widetilde w_2)$};
				\node at (.13,.295) {$\in p^{-1}(x)$};
				\node at (-.05,.775) {$ p^{-1}(x) \ni$};		
			\end{scope}
		\end{tikzpicture}
		\caption{A curve $\beta$ cannot be simple if there is a lift $\widetilde \beta$ that runs close to a lift $\widetilde \gamma$ of a non-simple curve $\gamma_i$ for a long enough time.} \label{F:non-simple}
	\end{figure}

	 Let $C$ be the constant from Lemma \ref{L:wide_rectangles_are_narrow} for the $r$ and $\epsilon$ from Claim \ref{CL:simple_cannot_follow_non-simple}. Suppose for the purposes of contradiction that $n \geq \max \{C+2,L+3\}$. Recall, $\delta_u$ and $\delta_G$ form a convex quadrilateral with with $\widetilde \alpha_1$ and $\widetilde \alpha_n$.
	 By choice of the hyperbolic metric on $S$,  $\widetilde \alpha_1$ and $\widetilde \alpha_n$ are at least $n-2 \geq C$ apart. Hence Lemma \ref{L:wide_rectangles_are_narrow} says there is a subsegment of $\delta_u$ that is contained in the $\epsilon$--neighborhood of $\delta_G$. As shown above, $n \geq L+2$ implies $ p_0(\delta_G) = c_i$ for some $i \in \{1,\dots, k\}$. On the other hand,  $p(\delta_u)$ is a simple curve because  $p(\partial \hull_u) \subseteq \Phi(u)$ as described in \S\ref{S:fibers and trees}. However, this contradicts Claim \ref{CL:simple_cannot_follow_non-simple}, so we must have $n < \max\{C+2,L+2\}$. Since $C$ and $L$ do not depend on $u$, setting $D_{\parallel} = \max\{C+2,L+3\}$ completes the proof of Lemma \ref{L:bounding parallel subtrees}.
\end{proof}

Armed with bounds on the diameter of the hull and parallel subtrees, we can now prove Proposition \ref{P:bounded diameter}.

\begin{proof}[Proof of Proposition~\ref{P:bounded diameter} assuming Proposition \ref{P:weak bounded diameter}] Recall, we wish to prove a uniform bound $D$ on the diameter of $T_G \cap T_u$ for every simplex $u \subset \C(S^z)$.  We claim that setting $D = D_\hull + 2D_\parallel + 2$ suffices. 

By Lemma~\ref{L:options_for_T_u} parts \eqref{I:fixed_sets} and \eqref{I:non-surviving}, we can assume $u \subset \C^s(S^z)$ and $T_u$ has infinite diameter, while Lemmas \ref{L:translating suffices} and  \ref{L:translating deep} say  it suffices to bound $\diam(T_G \cap T_u)$  when $u$ is a deep simplex.  Let $t,t' \in T_G \cap T_u$ be any two vertices.  The geodesic, $\ell$, connecting $t$ and $t'$ decomposes into at most five segments, two contained in parallel subtrees, one in the hull subtree, and a pair of edges connecting the segments in parallel subtrees to the segment in the hull subtree.  It follows from Lemmas~\ref{L:bounding the hull subtree} and \ref{L:bounding parallel subtrees} that the length of $\ell$ is at most $D$.  Since $t,t'$ were arbitrary, this completes the proof. 
\end{proof}

\section{Bounding the diameter of $p_0(T_u \cap T_G)$.}\label{S:Bounding_intersection_in_sigma_0}

The goal of this section is to prove Proposition~\ref{P:weak bounded diameter}, which asserts the existence of a uniform bound $D'$ on the diameter of $p_0(T_u \cap T_G)$ in $\sigma_0$.  As shown in the previous section, this will complete the proof of Proposition~\ref{P:bounded diameter} and hence Theorem~\ref{T:reducible}.

Recall from \S\ref{S:set up} that $\alpha = \alpha_1 \cup \ldots \cup \alpha_n$ is the canonical reduction system of the pure, reducible homeomorphism $f$ with complementary subsurfaces $Y_1,\ldots,Y_k$.  
For each edge $e$ of $T$ (or $\sigma_0$), we write $A_e = A_{i(e)}$ for the annular cover of $S$ corresponding to the component $\alpha_e = \alpha_{i(e)} \subset \alpha$; if $e$ is an edge in $T$, then $A_e = \mathbb H^2/K_e$.  If $e \subset T$ and $g \in \Gamma$, then we have a canonical identification $A_e = A_{g(e)}$. Likewise, for each vertex $t$ of $T$ (or $\sigma_0$), we write $Y_t = Y_{j(t)}$ for the corresponding complementary subsurface of $S$, given by $Y_t = \widetilde Y_t/K_t$ for $t \in T$.  If $t \in T$, $g \in \Gamma$, then $Y_t= Y_{g(t)}$.  Observe that $f$ acts on each $\AC(Y_j)$ by restricting $f|_{Y_j}$ and on each $\A(A_i)$ by lifting $f$ to $A_i$. 

We say that a vertex $t \in T$ is a \emph{pseudo-Anosov vertex} (resp \emph{identity vertex}) of $T$ if $f$ acts by a pseudo-Anosov (resp.~by the identity) on $Y_{t}$; that is, if $Y_{t}$ is a pseudo-Anosov (resp.~identity) component of $f$.  Recall, by \cite{MM1} $f$ will act by a pseudo-Anosov on $Y_{t}$ if and only if $f$ acts loxodromically on $\AC(Y_{t})$. We say  $e$ is a \emph{twist edge} of $T$ if  $f$  acts loxodromically on $\A(A_e)$. This occurs if the complementary components $Y_j$,$Y_{j'}$ of $\alpha$ that meet the curve $\alpha_{e}$ are identity components, and hence $f$ acts by a power of a Dehn twist in $\alpha_{e}$, or if at least one of $Y_j$ or $Y_{j'}$ are pseudo-Anosov components which effect (possibly fractional) non-canceling Dehn twists about the boundary component(s) corresponding to $\alpha_{e}$.  Since the assignments of  $Y_t$ and $A_e$ are $G$--equivariant, the labeling of pseudo-Anosov/identity vertex and twist edge are also $G$--equivariant. Hence they descend to give the same labels to vertices and edges of $\sigma_0 = T_G /G_0$. 
The next lemma ensures that every path of length 2 in $T$ contains either a twist edge or a pseudo-Anosov vertex.

\begin{lemma} \label{L:ubiquitous} For any non-twist edge, $e \subset T_G$, at least one endpoint is a pseudo-Anosov vertex.
\end{lemma}
\begin{proof} If neither endpoint of $e$ is a pseudo-Anosov vertex, then $f$ must act as a Dehn twist about $\alpha_e$, since otherwise $\alpha_e$ would not be in the canonical reduction system for $f$. Thus, $e$ is a twist edge.
\end{proof}

\subsection{Edge and vertex decorations} \label{S:projections}

To each edge $e$ and vertex $t$ of $T_G$ we will assign a bounded diameter subset $\Delta_e$ and $\Delta_t$ of the arc and curve graph of $A_e$ and $Y_t$, respectively.  We call these {\em decorations} of the edges and vertices.

For each edge $e$ of $T_G$, there are exactly two geodesics in $\partial \hull_G$ that non-trivially intersect $\widetilde \alpha_e$.  Define $\Delta_e \subset \A(A_e)$ to be the union of the images of these two geodesics under the covering map $\mathbb H^2 \to A_e$. If $e$ and $e'$ are edges of $T_G$ that are in the same $G_0$--orbit, then $\Delta_e = \Delta_{e'}$ because   $G_0$ preserves $\hull_G$ and $A_e = A_{e'}$ .

For each vertex $t$ in $T_G$, each geodesic arc $\widetilde \gamma$ in $\widetilde Z_t \subset \widetilde Y_t$ with endpoints in $\partial_\alpha \widetilde Z_t =  \widetilde Z_t \cap p^{-1}(\alpha)$ projects to a geodesic path $\gamma$ in $Y_t$; see Figure~\ref{F:Y and Z}. For each such path $\gamma$, we consider the self-intersection number ${\mathbb I}(\gamma)$, which is the minimum number of double points of self intersection over all representatives of the homotopy class rel endpoints (which is realized by the unique geodesic representative orthogonal to the boundary).  For each $t$, there are only finitely many homotopy classes of such arcs, $\gamma_1,\ldots,\gamma_{r(t)}$, and we set
\[ \Delta_t = \{\beta \in \AC(Y_t) \mid i(\beta,\gamma_j) \leq 2 \mathbb I(\gamma_j) \mbox{ for some } j \in \{1,\ldots,r(t)\}\}.\]
Note that by taking a representative of $\gamma_j$ with only double points of self intersection realizing $\mathbb I(\gamma_j)$, we can construct an arc $\beta_j$ in $Y_t$ from surgery on these self intersection points, and then pushing off, so that $i(\beta_j,\gamma_j) \leq 2 \mathbb I (\gamma_j)$. In particular, $\Delta_t \neq \emptyset$.  Moreover, any $\beta$ with $i(\beta,\gamma_j) \leq 2 \mathbb I(\gamma_j)$ also has $i(\beta,\beta_j) \leq 2 \mathbb I(\gamma_j)$ since $\beta_j$ is constructed from arcs of $\gamma_j$.  Since distance is bounded by a  function of intersection number (see e.g.~\cite{MM1}), it follows that $\Delta_t$ has finite diameter in $\AC(Y_t)$.
As with the edge decorations, if $t$ and $t'$ are vertices in the same $G_0$--orbit, then $\Delta_t = \Delta_{t'}$.

\begin{center}
	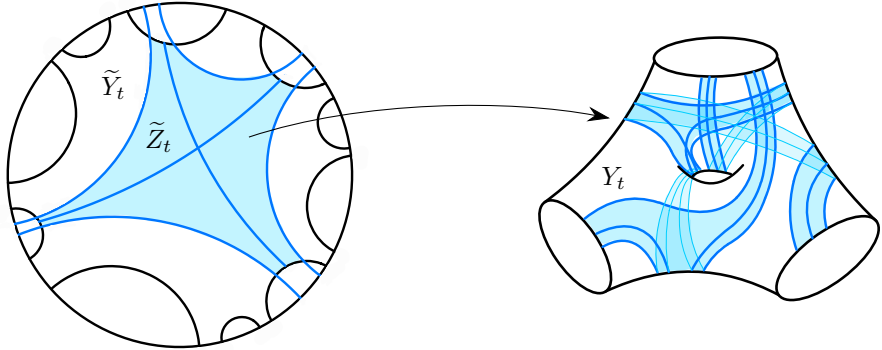
\begin{figure}[htb]
		\begin{tikzpicture}
			\node at (0,0) {\includegraphics[width=12cm]{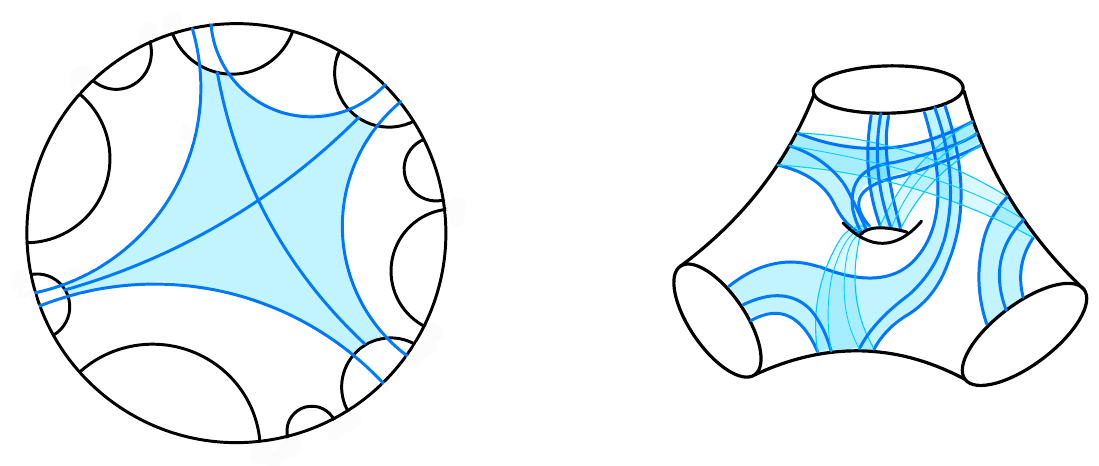}};
			\node at (-3.7,.5) {$\widetilde Z_t$};
			\node at (-4.3,1.2) {$\widetilde Y_t$};
			\node at (2.35,-.05) {$Y_t$};
			\draw[-\tips] (-2.5,.5) .. controls (-1,1) and (1,1) .. (2.3,.75);
		\end{tikzpicture}
		\caption{Left: The polygon $\widetilde Z_t \subset \hull_G$ (shaded) contained in $\widetilde Y_t$ and essential geodesics segments contained in it.  Right: The image of $\widetilde Z_t$ and its arcs in $Y_t = \widetilde Y_t/K_t$.} \label{F:Y and Z}
	\end{figure}
\end{center}

The next lemma describes how these decorations behave under arbitrary elements of $G$.   Recall that $\phi \colon G \to \mathbb Z$ is the homomorphism so that $ f^{\phi(g)} = \Phi_*(g)$ for any  $g \in G$.

\begin{lemma} \label{L:equivariant decoration} For any edge $e$ or vertex $t$ of $T_G$ and $g \in G$, we have 
	\[ \Delta_{g(e)} =f^{\phi(g)} (\Delta_e) \quad \mbox{ and } \quad  \Delta_{g(t)} =  f^{\phi(g)}(\Delta_t).\]
\end{lemma}
\begin{proof}
	Observe that each $g \in G$ maps each geodesic of $\partial \hull_G$ to a bi-infinite path that is homotopic, rel the ideal endpoints, to a geodesic in $\partial \hull_G$ (since these are completely determined by the components of $p^{-1}(\alpha)$ that are intersected).  Since $g$ descends to the lift of $ f^{\phi(g)}$ on each $A_e = A_{g(e)}$, the first equation follows.
	
	For the second equation, let $\widetilde \gamma \subset \widetilde Z_t$ be any geodesic arc with endpoints in $\partial_\alpha \widetilde Z_t$ and $\gamma$ the image path in $Y_t$.  
	Next, observe that $g$ descends to the restriction of $ f^{\phi(g)}$ to $Y_t = Y_{g(t)}$, and so maps $\gamma$ to a path $ f^{\phi(g)}(\gamma)$, which is homotopic to the image of a geodesic in $\widetilde Z_{g(t)}$.  Therefore, the restriction of $ f^{\phi(g)}$ to $Y_t$ maps the finite set of homotopy classes of paths defining $\Delta_t$ to those defining $\Delta_{g(t)}$, and hence sends $\Delta_t$ to $\Delta_{g(t)}$.
\end{proof}

As a consequence, we have

\begin{corollary} \label{C:bounded projection}
	There exists a constant $B_0 >0$ so that
	\[ \diam(\Delta_e),\diam(\Delta_t) \leq B_0 \]
	for all vertices $t$ and edges $e$ of $T_G$.
\end{corollary}
\begin{proof}
	There are only finitely many $G$--orbits of edges and vertices in $T_G$ and for any $g \in G$, $ f^{\phi(g)}$ acts by simplicial automorphisms on $\A(A_e)$ and $\AC(Y_t)$ for every edge $e$ and vertex $t$.  By Lemma~\ref{L:equivariant decoration}, it follows that
	\[ \diam(\Delta_{g(e)}) = \diam(f^{\phi(g)}( \Delta_e)) = \diam(\Delta_e) \quad \mbox{ and } \quad \diam(\Delta_{g(t)}) = \diam( f^{\phi(g)}(\Delta_t)) = \diam(\Delta_t).\]
Therefore, we can take $B_0$ to be the maximum diameter of $\Delta_e$ and $\Delta_t$ taken over a finite set of $G$--orbit representatives of edges $e$ and vertices $t$.
\end{proof}

Since $\Delta_e = \Delta_{e'}$ and $\Delta_t = \Delta_{t'}$ for $e,e'$ or $t,t'$ in the same $G_0$--orbit, these 
decorations on edges and vertices descend to decorations on the edges and vertices of $\sigma_0 =T_G/G_0$. We denote these by $\Delta_\ee$ and $\Delta_\tau$ for an edge $\ee$ or vertex $\tau$ of $\sigma_0$. Since $\frak f$ is isotopic to a lift of $f$ to $S_0$ preserving $\Sigma_0$ and its spine $\sigma_0$, the action of $G/G_0 \cong \langle \frak f \rangle$ on $\sigma_0 = T_G/G_0$ and $\Sigma_0 = \hull_G/G_0$ also induces an action on the decorations, satisfying the analogous formula to Lemma \ref{L:equivariant decoration}:
\begin{equation} \Delta_{\frak f^n(\ee)} = f^n(\Delta_\ee) \quad \mbox{ and } \quad \Delta_{\frak f^n(\tau)} = f^n(\Delta_\tau)
\end{equation} 
for every edge $\ee$ and vertex $\tau$ of $\sigma_0$ and every $n \in \mathbb Z$.

\subsection{Projections} \label{S:MM projections}
Given a multicurve $v \subset \C(S)$, Masur and Minsky defined a  projection of $v$ to the arc and curve graph of subsurfaces and annular covers of $S$ \cite{MM2}. We will describe these projection in the special cases of $A_e$ and $Y_t$.

For each vertex $t \in T_G$, the multicurve $v$ intersects $Y_t$ in a collection of disjoint curves and arcs, producing a (possibly empty) simplex of $\AC(Y_t)$. Let $\pi_t(v) \subset \AC(Y_t)$ be this simplex.  We observe that $\pi_t(v)$ is precisely the set of essential arcs and curves that are in the image of $p^{-1}(v) \cap \widetilde{Y}_t$ under the covering map $\widetilde Y_t \to Y_t$ (compare with the definition of $\Delta_t$).  Since $Y_t = Y_{t'}$ if $t$ and $t'$ are in the same $G$--orbit, we have $\pi_t(v) = \pi_{t'}(v)$ in this case.

For an edge $e \subset T_G$, we define $\pi_e(v) \subset \A(A_e)$ to be the set of essential arcs in the preimage of $v$ under the covering map $A_e \to S$. As in the case of $\pi_t$, we note that $\pi_e(v)$ is precisely the essential arcs in the image of $p^{-1}(v)$ under the covering map $\mathbb H^2 \to A_e$ (compare with the definition of $\Delta_e$). Since $v$ is a collection of disjoint curves, $\pi_e(v)$ is a simplex of $\A(A_e)$. Recall, the core curve of $A_e$ is (a lift of) one of the curves $\alpha_e$ in $\alpha$. Thus, we have $\pi_e(v) \neq \emptyset$ if and only if $i(v,\alpha_e) \neq 0$.   Since $A_e = A_{e'}$ when $e$ and $e'$ are in the same $G$--orbit, we have $\pi_e(v) = \pi_{e'}(v)$ for such pairs of edges.

Since $A_e$ and $Y_t$ are determined by the $G$--orbit of the edge or vertex, we can define projection for vertices and edges of $\sigma_0$ by \[ \pi_\ee(v) = \pi_e(v) \text { and } \pi_\tau(v) = \pi_t (v)\] where $\ee = p_0(e)$ and $\tau = p_0(t)$.

Given an edge $e$ or vertex $t$ of $T_G$ (or $\sigma_0$), we let $d(\Delta_e, \pi_e(v))$ and $d(\Delta_t,\pi_t (v))$ denote the diameter of  $\Delta_e \cup \pi_e(v)$ and $\Delta_t \cup\pi_t (v)$ in $\A(A_e)$ and $\AC(Y_t)$, respectively. 
Our proof of Proposition \ref{P:weak bounded diameter} hinges upon understanding for how many vertices/edges in a row these diameters can be large along a path in $T_G$.

  For an edge $\ee \subset \sigma_0$ or vertex $\tau \in \sigma_0$ and $B >0$, define:
\[ \mathcal E(v,B) = \{ \ee \subset \sigma_0 \mid \ee \mbox{ is a twist edge of } \sigma_0, \pi_\ee (v) \neq \emptyset, \mbox{ and } d(\Delta_\ee,\pi_\ee (v)) \leq B \} \]
\[ \mathcal V(v,B)= \{ \tau \in \sigma_0 \mid \tau \mbox{ is a pA vertex of } \sigma_0, \pi_\tau(v) \neq \emptyset, \mbox{ and } d(\Delta_\tau,\pi_\tau(v)) \leq B \}\]
We view these both as sets of edges and vertices, respectively, and as subgraphs of $\sigma_0$ (defined by taking the union of the corresponding set of edges/vertices).
Let $V,E >0$ be the numbers of $G$--orbits of vertices and edges in $T_G$, respectively.  Equivalently, $V,E$ are the numbers of $\langle \frak f \rangle$--orbits of vertices and edges in $\sigma_0$, respectively.
\begin{lemma} \label{L:bounded clusters} For any $B > 0$ there exists $M > 0$ so that the following holds for each multicurve $v \subset \C(S)$: 
	\begin{enumerate}
		\item $\mathcal E(v,B)$ is a union of at most $E$ sets of diameter at most $M$. \label{I:E_bound}
		\item $\mathcal V(v,B)$ is a union of at most $V$ sets of diameter at most $M$. \label{I:V_bound}
	\end{enumerate}
\end{lemma}
\begin{proof} 
	Fix $B \geq 0$ and a multicurve $v \subset \C(S)$.   For part \eqref{I:E_bound}, it suffices to fix a twist edge $\ee \subset \sigma_0$ and bound the diameter of the subset of $\mathcal E(v,B)$ consisting of edges in the set $\langle \frak f \rangle \cdot \ee$ by a constant $M_\ee$, independent of $v$.  
	
	For this, suppose $\pi_\ee (v) \neq \emptyset$. Since the monodromy $f$ fixes the annulus $A_\ee$, we have $\pi_{\frak{f}^n(\ee)}(v) = \pi_\ee (v)$ for all $n \in \mathbb Z$. Moreover, Lemma \ref{L:equivariant decoration} says $\Delta_{\frak{f}^n (\ee)} = f^n(\Delta_\ee)$.  Therefore $$d(\Delta_{\frak{f}^n(\ee)},\pi_{\frak{f}^n(\ee)}(v)) =d(f^n(\Delta_{\ee}), \pi_\ee (v))$$ for each $n \in \mathbb{Z}$. 
	
	Since $f$ acts loxodromically on $\A(A_\ee)$, the set of integers $n$ for which $f^n(\Delta_\ee)$ can intersect the $B$--neighborhood of $\pi_\ee (v)$ is contained in an interval of integers $I_\ee \subset \mathbb Z$ whose width, $W_\ee$, depends only on $B$ and the loxodromic constants of the action of $f$ on $\A(A_\ee)$ (and in particular, it is independent of $v$).  Thus we have
	\[ \langle \frak f \rangle \cdot \ee \cap \mathcal E(v,B) \subseteq \bigcup_{n \in I_\ee } \frak f^n(\ee). \]
	The union on the right has diameter at most $W_\ee$ times the distance in $\sigma_0$ between $\ee$ and $\frak f(\ee)$ (or equivalently, the distance between $\frak f^n(\ee)$ and $\frak f^{n+1}(\ee)$, for any $n \in \mathbb Z$).  This bound thus also bounds the diameter of $\langle \frak f \rangle \cdot \ee \cap \mathcal E(v,B)$, and taking any $M$ which is at least the maximum such bound over all $E$ orbit representatives of edges, implies part \eqref{I:E_bound}.
	
	The proof for part \eqref{I:V_bound} is nearly identical, choosing a pseudo-Anosov vertex and using $Y_\tau$ instead of $A_\ee$ and the fact that $f$ acts loxodromically on $\AC(Y_\tau)$.
\end{proof}

The following is an immediate corollary of Lemma~\ref{L:bounded clusters} plus the bound on the valence of $\sigma_0$ (or more directly from the proof).
\begin{corollary}\label{C:finitely_many_small_labels}
	For each $B \geq 0$ there exists $N \geq 0$ so that for each multicurve $v \subset \C(S)$, we have
	\[ |\mathcal E(v,B) | \leq N \quad \mbox{ and } \quad |\mathcal V(v,B)| \leq N. \]
\end{corollary}

\subsection{Parallel type subtrees proof}

 We now prove that the image under $p_0$ of any parallel subtree of $T_G \cap T_u$ is uniformly bounded. The main fact we need is that large vertex and edge projection can only occur along the leaves of the parallel subtrees.

 \begin{lemma}\label{L:small_proj_parallel_type}
 	There exists $B_1\geq 0$ so that the following holds for each multicurve $u \subset \C^s(S^z)$.
 	\begin{enumerate}
 		\item \label{I:vertex} 	Let $\ell$ be an edge path of length 2 in $T_G \cap T_u$, $t$ be the middle vertex of $\ell$, and $v= \Phi(u)$.   If each vertex of $\ell$ is of parallel type, then $d(\Delta_t,\pi_t (v)) \leq B_1.$ 
 		\item \label{I:edge} Let   $\ell$ be an edge path of length 3 in $T_G \cap T_u$, $e$ be the middle edge of $T_G \cap T_u$, and $v= \Phi(u)$.  If each  vertex of $e$ is of parallel type, then  $d(\Delta_e,\pi_e (v)) \leq B_1.$ 
 	\end{enumerate}
 	
 \end{lemma}
 
 \begin{proof}
 	First let $\ell$ be path of length 2 in $T_G \cap T_u$. Let $e_1$, $e_2$ be the edges of $\ell$ and $t$ be the middle vertex.  Let  $\widetilde \alpha_i$ be the geodesic of $p^{-1}(\alpha)$ that is dual to the edge $e_i$. If each vertex of $\ell$ is of parallel type, then Lemma \ref{L:Parallel_type_vertices} says there exist geodesics  $\delta_G \subseteq \partial \hull_G$ and $\delta_u \subseteq \partial \hull_u$ that intersect $\widetilde \alpha_1$ and $\widetilde \alpha_2$, but do not intersect in $\widetilde{Y}_{t}$. Hence there is a straight line homotopy  relative $\partial \widetilde{Y}_t$ of $\delta_G \cap \widetilde{Y}_t$ to $\delta_u \cap \widetilde{Y}_t$. Since $\delta_u$ and $\delta_G$ intersect the same components of $\partial \widetilde{Y}_t$, this straight line homotopy descends to a  homotopy relative $\partial Y_t$ of $p(\delta_G \cap \widetilde{Y}_t)$ to $p( \delta_u \cap \widetilde{Y}_t)$.  In particular, $p(\delta_G \cap \widetilde Y_t)$ is an arc on $Y_t$  that is equal to $p(\delta_u \cap \widetilde Y_t)$ as an element of $\AC(Y_t)$. Since $p(\delta_u\cap \widetilde{Y}_t) \subseteq \pi_t (v)$ and  $p(\delta_G \cap \widetilde{Y}_t) \subseteq \Delta_t$, part \eqref{I:vertex} of the lemma now follows from Corollary \ref{C:bounded projection} for any $B_1 \geq B_0 + 1$.

 	Now, let $\ell$ be a path of length 3 in $T_G \cap T_u$. Let $e_1,e_2,e_3$ be the edges of $\ell$ and  $\widetilde \alpha_i$ be the geodesic of $p^{-1}(\alpha)$ that is dual to the edge $e_i$. If each vertex of $\ell$ is of parallel type, then Lemma \ref{L:Parallel_type_vertices} says there exist geodesics  $\delta_G \subseteq \partial \hull_G$ and $\delta_u \subseteq \partial \hull_u$ that intersect each of  $\widetilde \alpha_1$, $\widetilde \alpha_2$, and $\widetilde \alpha_3$, but do not intersect each other between $\widetilde \alpha_1$ and $\widetilde \alpha_3$. Now,  $\widetilde \alpha_1$ and $\widetilde \alpha_3$ determine a unique convex ideal rectangle $R$; see Figure \ref{F:rectangle}. Let $\beta_1$ and $\beta_2$  be the other two sides of $R$.
 	
 	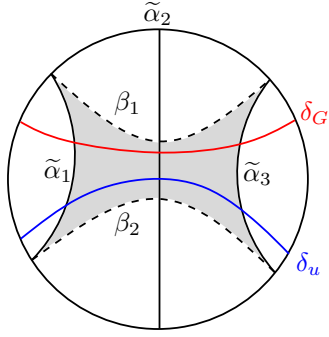
\begin{figure}[htb]
 		\begin{tikzpicture}
 			\node[anchor=south west,inner sep=0] (image) at (0,0) {\includegraphics[width=4cm]{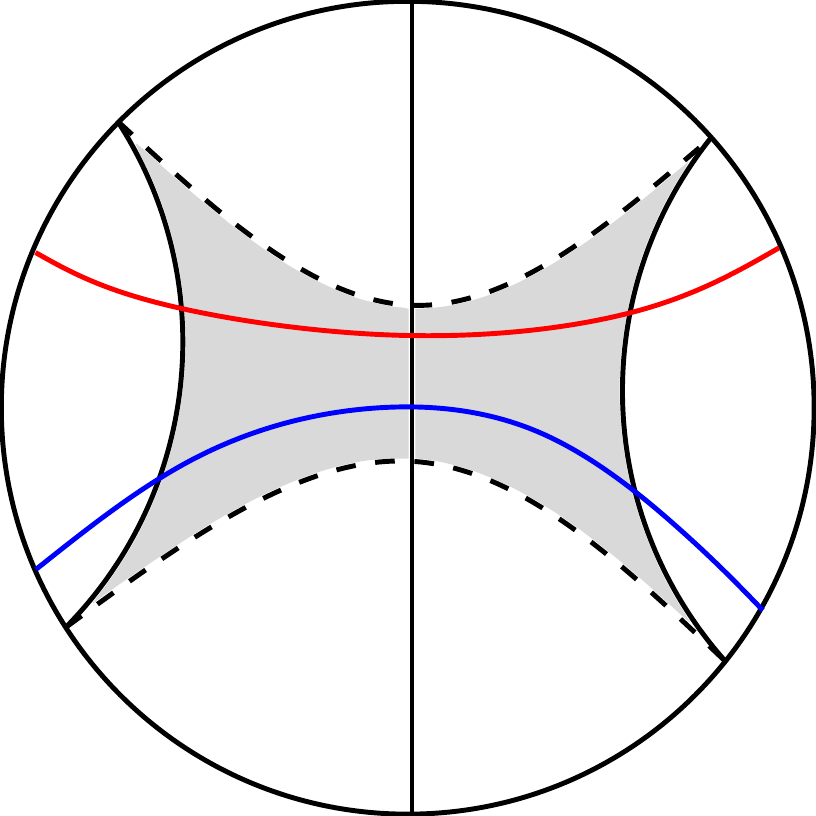}};
 			\begin{scope}[x={(image.south east)},y={(image.north west)}]
 				\node at (.17,.53) {$\widetilde \alpha_1$};
 				\node at (.5,1.05) {$\widetilde \alpha_2$};
 				\node at (.83,.52) {$\widetilde \alpha_3$};
 				\node at (1.02,.72) {\color{red} $\delta_G$};
 				\node at (1,.22) {\color{blue} $\delta_u$};
 				\node at (.4,.75) {$\beta_1$};
 				\node at (.4,.35) {$\beta_2$};
 			\end{scope}
 		\end{tikzpicture}
 		\caption{The shaded rectangle $R$ is the unique ideal rectangle  with sides $\widetilde \alpha_1$ and $\widetilde \alpha_3$.} \label{F:rectangle}
 	\end{figure}

 	We claim that for each non-identity element $g \in K_{e_2}$, we have $g  (R) \cap R = \emptyset$. If this is true, then the rectangle $R$ will embed onto the annulus $A_{e_2} = \mathbb{H}^2 /K_{e_2}$.  In particular, the images of $\delta_u$ and $\delta_G$ would be disjoint from the image of $\beta_1$ in $A_{e_2}$. Since  $\delta_G \subseteq \partial \hull_G$ and $\delta_u \subseteq \partial \hull_u$, these project to arcs in $\Delta_{e_2} \subset \A(A_{e_2})$ and $\pi_{e_2}(v) \subset \A(A_{e_2})$, respectively,  which are distance at most $2$ apart in $\A(A_{e_2})$.  Therefore,
	\[ d(\Delta_{e_2},\pi_{e_2}(v)) \leq 2 + B_0 + 1= B_0 + 3\]
by Corollary~\ref{C:bounded projection}, and thus, setting $B_1 = B_0+3$ proves part \eqref{I:edge} of the lemma.
 	
 	It therefore remains to prove that $g  (R) \cap R = \emptyset$ for each non-trivial $g \in  K_{e_2}$. First recall that  $K_{e_2}=\Stab_{\pi_1 S}(\widetilde \alpha_2)$ is a cyclic group generated by a hyperbolic isometry of $\mathbb H^2$ with axis $\widetilde \alpha_2$. Since $\widetilde \alpha_1$ and $\widetilde \alpha_3$ lie on different sides of $\widetilde \alpha_2$, the only way for $g  (R) \cap R\neq \emptyset$ is for  $g \widetilde (\alpha_i) \cap \widetilde \alpha_i \neq \emptyset$ for either $i=1$ or $i=3$ (or both). However, no two geodesics in $p^{-1}(\alpha)$ intersect because $\alpha$ is a collection of disjoint simple closed curves, and every non-identity element of $K_{e_2}$ takes every element of $p^{-1}(\alpha) - \{\widetilde \alpha_2\}$ to a different element of  $p^{-1}(\alpha) - \{\widetilde \alpha_2\}$. Together these imply that $g  \widetilde (\alpha_i) \cap \widetilde \alpha_i = \emptyset$ for $i \in \{1,3\}$ and $g \in K_{e_2}$ not equal to the identity.  Hence $g (R) \cap R = \emptyset$ for each non-trivial $g \in K_{e_2}$.  This proves the claim above, and hence the lemma.
 \end{proof}
 
 Combining Lemma \ref{L:small_proj_parallel_type} with Corollary \ref{C:finitely_many_small_labels} will produce the desired bound on the images of the parallel subtrees.

 \begin{lemma}\label{L:bound_on_parallel_type_diam}
 	There exists $D_1 \geq 0$ so that  for any multicurve $u \subset \C(S^z)$, if $P$ is a parallel subtree of $T_G \cap T_u$, then $\diam(p_0 (P)) \leq D_1$.
 \end{lemma}

 \begin{proof} Let $N$ be the constant from Corollary~\ref{C:finitely_many_small_labels} for $B = B_1$.  We claim that taking $D_1 = 3N+6$ will suffice to prove the lemma.
 	
	Let $\tau = p_0(t)$ and $\tau' =p_0(t')$ be vertices of $p_0(P)$. Let  $\ell$ be a geodesic path in $P$ from $t$ to $t'$and let $\gamma$ be the path in $\sigma_0$ that is the image of $\ell$ under $p_0$.  Let $t=t_0,\dots, t_n = t'$ be the vertices of $\ell$ and let $e_i$ be the edge of $\ell$ between $t_{i-1}$ and $t_{i}$. Let $v = \Phi(u)$. By Lemma \ref{L:small_proj_parallel_type}, we have \[ d(\Delta_{e_i},\pi_{e_i} (v))\leq B_1 \text{ and }d(\Delta_{t_j},\pi_{t_j}(v)) \leq B_1\] for each $i \in \{2,\dots, n-2\}$ and $j \in \{1,\dots,n-1\}$. Since $$d(\Delta_{e_i},\pi_{e_i} (v)) = d(\Delta_{p_0(e_i)},\pi_{p_0(e_i)}(v)) \text{ and } d(\Delta_{t_j},\pi_{t_j}(v)) = d(\Delta_{p_0(t_j)},\pi_{p_0(t_i)}(v)),$$   Corollary \ref{C:finitely_many_small_labels} implies that the path $\gamma \subset \sigma_0$ contains at most $N+2$ distinct twist edges and at most $N+2$ distinct pseudo-Anosov vertices.  Lemma \ref{L:ubiquitous} implies that every edge  of  $\gamma$ is either a twist edge or has  a pseudo-Anosov vertex as an endpoint. Hence, $\gamma$ is contained in \[ \bigcup \{N_1(\tau) \mid {\tau \text{ a pA vertex of } \gamma}\} \cup \bigcup\{ \ee \mid {\ee \text{ a twist edge of }\gamma}\}. \] Since the diameter of  $N_1(\tau)$  is at most 2, and there are at most $N+2$ pseudo-Anosov vertices and twist edges, we have $$\diam(\gamma) \leq 2(N+2) + (N+2) = 3N+6.$$ Since $\gamma = p_0(\ell)$ and $\ell$ is an arbitrary path in $P$, this implies $\diam(p_0(P)) \leq 3N+6 = D_1$ as desired. 
 \end{proof} 

\subsection{Hull type subtree proof}

Recall that $T^\hull_{u,G} \subset T_u \cap T_G$ is the hull subtree, as defined in \S\ref{S:subtrees}.  In this subsection we prove the following.

\begin{lemma} \label{L:hull subtree bound} There exists $D_2 > 0$ so that if $u \subset \C^s(S^z)$ is a simplex, then $\diam(p_0(T^\hull_{u,G})) \leq D_2$.
\end{lemma}

The first ingredient in the proof of this lemma is the following.

\begin{lemma} \label{L:dead ends} There exists a constant $B_2 \geq B_0$ with the following property.  For any multicurve $u \subset \C^s(S^z)$, edge $e \subset T^\hull_{u,G}$, and vertex $t \in T^\hull_{u,G}$, we have the following.
	\begin{enumerate}
		\item \label{I:edge bound} If $d(\Delta_e,\pi_e (\Phi(u))) > B_2$, then $T^\hull_{u,G} = e$.
		\item \label{I:vertex bound} If $d(\Delta_t,\pi_t (\Phi(u))) > B_2$, then $t$ is a valence $1$ vertex of $T^\hull_{u,G}$.
	\end{enumerate}
\end{lemma}
\begin{remark} Note that since $B_2 \geq B_0$, it follows that $d(\Delta_e,\pi_e (v)) > B_2$ implies $\pi_e (v) \neq \emptyset$ by Corollary~\ref{C:bounded projection}.  Similarly, if $d(\Delta_t,\pi_t (v)) > B_2$ then $\pi_t (v) \neq \emptyset$.
\end{remark}

\begin{proof} We start with the proof of part \eqref{I:vertex bound}, which is more direct.  For this, it suffices to take any $B_2 \geq B_0$.  To see this, suppose $u \subset \C^s(S^z)$ is any multicurve, $v = \Phi(u)$, and $t \in T^\hull_{u,G}$ is a vertex with valence at least $2$.  This means that $\hull_u \cap \hull_G$ must intersect two distinct components $\widetilde \alpha_e,\widetilde \alpha_{e'} \subset p^{-1}(\alpha)$, where $e,e'$ are adjacent to $t$.  Let $\widetilde \gamma$ be a geodesic arc from $\widetilde \alpha_e$ to $\widetilde \alpha_{e'}$ contained in $\hull_u \cap \hull_G$.  Observe that $\widetilde \gamma$ is therefore disjoint from $p^{-1}(v)$, and hence the image, $\gamma$, in $Y_t$ is disjoint from the intersection of $v$ with $Y_t$.  But then, every arc and curve in $\pi_t(v)$ is disjoint from $\gamma$ which implies $\pi_t(v) \subset \Delta_t$.  
	This means that $d(\Delta_t,\pi_t (v)) \leq \diam(\Delta_t) \leq B_0 \leq B_2$.  Thus, any vertex of $T^\hull_{u,G}$ with $d(\Delta_t,\pi_t (v)) > B_2$ must have valence $1$, proving \eqref{I:vertex bound}.
	
	We now explain part \eqref{I:edge bound}.   Fix an edge $e \subset T_G$ and let $\delta_1,\delta_2$ be the two geodesics in $\partial \hull_G$ meeting the geodesic $\widetilde{\alpha}_e$ dual to $e$.  The basic idea is that if $d(\Delta_e, \pi_e(v))$ is large, then there must be many geodesics of $p^{-1}(v)$ that cross both $\delta_1$ and $\delta_2$ and $\widetilde \alpha_e$, so that for any $u$ with $\Phi(u) = v$, a component of $\hull_u \cap \hull_G$ that meets $\widetilde \alpha_e$ is trapped in a bounded region.  We now proceed to the proof, and refer the reader to Figure~\ref{F:twisting figure} to aid in the argument.
	
	First observe that there are at most four edges $e'$ adjacent to $e$ such that at least one of $\delta_1$ or $\delta_2$ crosses the geodesic $\widetilde \alpha_{e'}$.  Write $\widetilde \alpha_i^j \subset p^{-1}(\alpha)$ for the geodesics corresponding to these edges so that $\delta_i$ non-trivially intersects $\widetilde \alpha_i^j$ for $i,j = 1,2$.

	Next, choose an $n$--sheeted covering $p_e \colon \widetilde A_e \to A_e$, for some $n > 0$, so that the projection $\mathbb H^2 \to \widetilde A_e$ is injective on the region bounded by $\delta_1$ and $\delta_2$ (which contains $\hull_G$)  as well as on each $\widetilde{\alpha}_i^j$ for $i,j =1,2$.  We denote the images of the $\delta_i$ and $\widetilde \alpha_i^j$ in $\widetilde A_e$ by the same name. We also use $p_e$ to denote the induced map between arc graphs $p_e \colon \A(\widetilde{A}_e) \to \A(A_e)$.  The core curve of $\widetilde A_e$ is an $n$--fold cover of the core curve $\alpha_e$ of $A_e$, and we denote it by $\alpha_e^n$.  See Figure~\ref{F:twisting figure}.  Observe that the degree $n$ necessary to arrange that all of these things happen can be chosen to depend only on the $G$--orbit of $e$, and since there are only finitely many $G$--orbits of edges in $T_G$, we can in fact assume that $n$ is independent of $e$.
	
	\begin{center}
		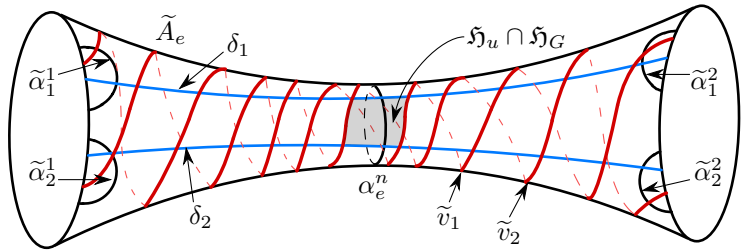
\begin{figure}[htb]
			\begin{tikzpicture}
				\node at (0,0) {\includegraphics[width=10cm]{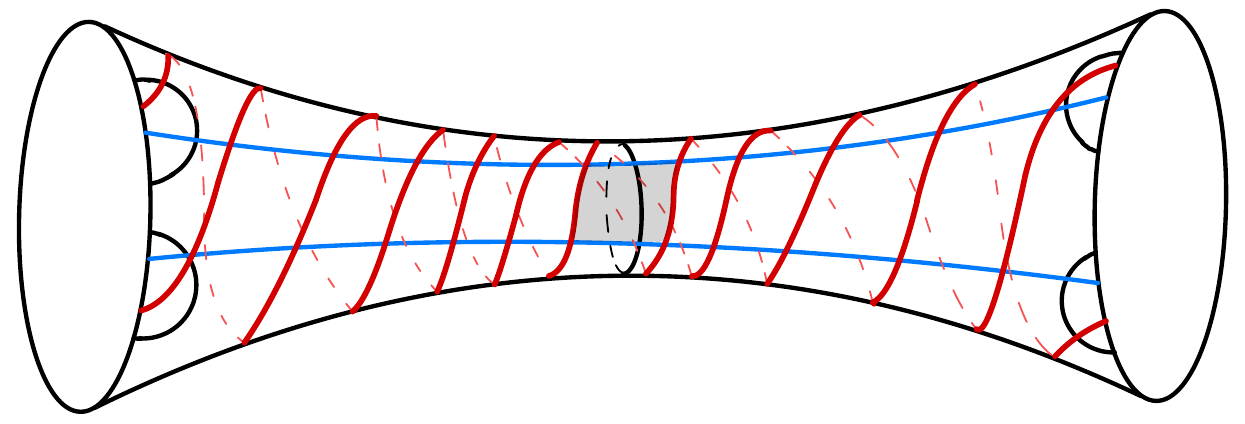}};
				\draw[-\tipss] (-2,1) -- (-2.6,.5);
				\node at (-1.8,1.1) {$\delta_1$};
				\draw[-\tipss] (-2.4,-1) -- (-2.5,-.26);
				\node at (-2.3,-1.2) {$\delta_2$};
				\draw[-\tipss] (4.2,.7) -- (3.61,.7);
				\node at (4.4,.6) {$\widetilde \alpha_1^2$};
				\draw[-\tipss] (4.2,-.7) -- (3.55,-.85);
				\node at (4.45,-.6) {$\widetilde \alpha_2^2$};
				\draw[-\tipss] (-4.2,.7) -- (-3.5,.9);
				\node at (-4.4,.6) {$\widetilde \alpha_1^1$};
				\draw[-\tipss] (-4.2,-.7) -- (-3.48,-.7);
				\node at (-4.4,-.6) {$\widetilde \alpha_2^1$};
				\node at (-2.7,1.3) {$\widetilde A_e$};
				\node at (0,-.8) {$\alpha_e^n$};
				\draw[-\tipss] (1.2,1.2) .. controls (.9,1.2) and (.4,.7)  .. (.25,.05);
				\node at (1.9,1.2) {$\hull_u \cap \hull_G$};
				\draw[-\tipss] (1,-1) -- (1.16,-.6);
				\node at (1,-1.2) {$\widetilde v_1$};
				\draw[-\tipss] (1.8,-1.2) -- (2.05,-.73);
				\node at (1.8,-1.4) {$\widetilde v_2$};
			\end{tikzpicture}
			\caption{A schematic of $\widetilde A_e$.} \label{F:twisting figure}
		\end{figure}
	\end{center}
	
	Next, observe that for any arc $\gamma \in \A(\widetilde A_e)$, we have
	\[ i(p_e(\gamma),p_e(\delta_i)) \leq n \, (i(\gamma,\delta_i)+1)\]
	for $i=1,2$.  Since distances in arc graphs of annuli are given by intersection number plus $1$, it follows that
	\[ d(p_e(\gamma),p_e(\delta_i)) \leq n \, d(\gamma,\delta_i) + 1,\]
	for $i=1,2$.  In particular, note that 
	\[ \begin{array}{rcl} d(\pi_e (v),\Delta_e) & = & d(\pi_e (v),p_e(\delta_1) \cup p_e(\delta_2)) \\
		& \leq & n \, d(p_e^{-1}(\pi_e (v)),\delta_1 \cup \delta_2)+1\\
		& \leq & n \, (d(p_e^{-1}(\pi_e (v)),\delta_i) + 1)+1, \end{array}\]
	for each $i=1,2$, since $d(\delta_1,\delta_2) = 1$ in $\A(\widetilde A_e)$.
	
	We set $B_2 \geq \max\{B_0,10n+1\}$.  Suppose $u \subset \C^s(S^z)$ is any multicurve for which $T^\hull_{u,G}$ contains $e$, $v = \Phi(u)$, and $d(\pi_e (v),\Delta_e) > B_2$.  Since the diameter of $p_e^{-1}(\pi_e (v))$ is $1$, it follows that for any $\widetilde v \in p_e^{-1}(\pi_e (v))$ and $i=1,2$, we have
	\[ d(\widetilde v,\delta_i)  \geq  d(p_e^{-1}(\pi_e (v)),\delta_i) - 1.\] 
 Combining this with the inequalities above gives
	\[ d(\widetilde v,\delta_i) \geq  \frac1n (d(\pi_e (v),\Delta_e) -1) -2 > \frac1n(B_2-1)-2 \geq \frac1n 10n-2 \geq 8. \]
	Thus, any $\widetilde v \in p_e^{-1}(\pi_e (v))$ intersects each of $\delta_1$ and $\delta_2$ in at least $7$  points in $\widetilde A_e$. Since  $\widetilde v$ and $\delta_i$ are geodesics in $\widetilde A_e$, the difference in the number of intersection points on the two sides of the core geodesic $\alpha_e^n$ is at most $1$.
	It follows that there are at least $3$ points of intersection of $\widetilde v$ with each of $\delta_1$ and $\delta_2$ on either side of $\alpha_e^n$.
	
	We now see that for any $\widetilde v \in p_e^{-1}(\pi_e (v))$ there are arcs of intersection of $\widetilde v$ in $\widetilde A_e$ with the region bounded by $\delta_1$ and $\delta_2$ that contains $\hull_G$, on both sides of $\alpha_e^n$.  Moreover, there are such segments that meet $\delta_1$ and $\delta_2$ between the geodesics $\{\widetilde{\alpha}_i^j\}_{i,j}$ (since once $\widetilde v$ meets $\widetilde{\alpha}_i^j$ it can intersect $\delta_i$ in at most one more point), and therefore, each segment is contained in the image of the corresponding $Z_t$, for $t$ an endpoint of $e$.
	
	Since $T^\hull_{u,G}$ contains $e$, the projection of $\hull_u \cap \hull_G$ to $\widetilde A_e$ necessarily intersects $\alpha_e^n$.  It is therefore contained in the region between two of the segments of any $\widetilde v$ described above.  In fact, it follows that there are $\widetilde v_1,\widetilde v_2 \in p_e^{-1}(\pi_e (v))$ such that $\hull_u \cap \hull_G$ is contained in the region bounded by these two geodesics together with $\delta_1$ and $\delta_2$; see Figure~\ref{F:twisting figure}.  This implies that $T^\hull_{u,G} = e$, as required.  This completes the proof of \eqref{I:edge bound}, and hence the lemma.
\end{proof}

The second ingredient is the following bound on paths that are not contained in the subsets  $\mathcal E(v,B)$ and $\mathcal V(v,B)$ from  \S\ref{S:projections}.   Recall that $E$ denotes the number of $G/G_0$--orbits of edges in $\sigma_0$.

\begin{lemma} \label{L:lengths outside} Let $B_2$ be as in Lemma~\ref{L:dead ends} and suppose $u \subset \C^s(S^z)$ is any multicurve, $v = \Phi(u)$, and $\gamma \subset T^\hull_{u,G}$ is an embedded edge path such that $p_0(\gamma) \subset \sigma_0$ is disjoint from $\mathcal E(v,B_2) \cup \mathcal V(v,B_2)$.  Then the length of $\gamma$ is at most $2E+2$.
\end{lemma}
\begin{proof} Suppose that the length of $\gamma$ is greater than $2E+2$ and let $\gamma_0 \subset \gamma$ be the subpath obtained by deleting the first and last edge.  Note this path contains no valence $1$ vertices of $T^\hull_{u,G}$, and consequently by Lemma~\ref{L:dead ends}, every vertex $t$ of $\gamma_0$ is either an identity (non-pseudo-Anosov) vertex or has $\pi_t (v) = \emptyset$. Since $T^\hull_{u,G}$ cannot be a single edge, another application of Lemma~\ref{L:dead ends} implies that for every edge $e$ of $\gamma_0$, either $e$ is a non-twist edge or $\pi_e (v) = \emptyset$.  In fact, we claim that something stronger holds for edges.
	
	\begin{claim} For every edge $e$ of $\gamma_0$, $\pi_e (v) = \emptyset$.
	\end{claim}
	\begin{proof} If $e$ is a twist edge, then we have already noted that $\pi_e (v) = \emptyset$, so it suffices to assume $e$ is a non-twist edge.  In this case, at least one of its endpoints, call it $t$, is a pseudo-Anosov vertex and hence $\pi_t (v) = \emptyset$.  Since $v$ cannot intersect the core curve $\alpha_e \subset A_e$ without intersecting $Y_t$, it follows that $\pi_e (v) = \emptyset$.
	\end{proof}

	Since the length of $\gamma_0$ is greater than $2E$, there must be a pair of edges $e_0,e_1$ of $\gamma_0$ so that $e_0$ and $e_1$---viewed as oriented edges, oriented by an orientation on $\gamma_0$---differ by an element $g \in G$.  Without loss of generality, suppose $e_0$ is the first of these edges encountered along $\gamma_0$.  Recall (see \S\ref{S:G0 quotient}) that we have embedded $T_G$ into $\hull_G \subset \mathbb H^2$, $G$--equivariantly on the vertices.  Using this, we let $\gamma_1$ be the subsegment of $\gamma_0$ that begins with $e_0$ and ends with $e_1$.  Let $\nu_1 \subset \gamma_1$ be the subpath starting from $e_0 \cap \widetilde \alpha_{e_0}$ and ending at $e_1 \cap \widetilde \alpha_{e_1}$.
	
	\begin{claim} The path $\nu_1$ is homotopic, rel endpoints, in $\mathbb H^2$ to a path $\nu_1'$ which is disjoint from $p^{-1}(v)$.
	\end{claim}
	\begin{proof} Since $\pi_e (v) = \emptyset$ for every edge $e$ of $\gamma_0$, it follows that $p^{-1}(v)$ is disjoint from $\widetilde \alpha_e$ for every such edge $e$. Let $t \in \nu_1$ be any vertex and let $e,e' \subset \gamma_1$ be the edges for which $t = e \cap e'$ with $e$ appearing before $e'$.  Let $\nu_t \subset \nu_1$ be the subpath from $e \cap \widetilde \alpha_e$ to $e' \cap \widetilde \alpha_{e'}$.  Either we can homotope $\nu_t$, rel endpoints, to a path disjoint from $p^{-1}(v)$, or else some component of $\widetilde v \subset p^{-1}(v)$ separates $\widetilde \alpha_e$ from $\widetilde \alpha_{e'}$.  The latter situation cannot happen, however, because then $\hull_u \cap \hull_G$ would have to lie on one side or the other of $\widetilde v$, contradicting the fact that $e$ and $e'$ are in $T^\hull_{u,G}$.  We can then combine the homotopies for each subpath $\nu_t$ associated to each vertex $t$ of $\nu_1$, producing the required homotopy for $\nu_1$.
	\end{proof}
	
	The rest of the proof splits into two cases.
	
	\medskip
	
	\noindent {\bf Case 1}. There is some vertex $t$ in $\nu_1$ for which $\pi_t (v) \neq \emptyset$.
	
	\medskip
	
	Note that $\pi_t (v) \neq \emptyset$ implies $t$ is an identity vertex.  If $\pi_t (v)$ contains a simple closed curve, then denote it by $w$ and note that it is a component of $v$.  Otherwise, $\pi_t (v)$ is a collection of arcs with endpoints on the boundary of $Y_t$. In this case, let $a$ be an arc of $\pi_t(v)$ and $c \subset \partial Y_t \subset \alpha$ be the component(s) of $\partial Y_t$ containing the endpoints of $a$. The boundary of a small neighborhood of $a \cup c$ contains either one or two essential curves on $Y_t$; see \cite[{\S 2}]{MM2}. Define $w$ to be one of these curves. 
	
	In either case,  by further homotopy if necessary, we may assume that $\nu_1'$ is disjoint from $p^{-1}(w)$.  This is obvious if $w \subset v$, while for the other case, we argue as follows. 
	Since $\pi_e (v) = \emptyset$ for every edge $e$ of $\gamma_0$, no component of $c$ is in $p(\widetilde \alpha_e)$ for any edge $e$ of $\gamma_0$. Hence, $p(\nu_1')$ is disjoint from both $c$ and the arc $a$, ensuring $\nu_1'$ is disjoint from $p^{-1}(w)$.
	
	Now, the element $g \in G$ maps $\widetilde \alpha_{e_0}$ to $\widetilde \alpha_{e_1}$.  Let $\nu_1''$ be the path obtained by concatenating $\nu_1'$ with an arc of $\widetilde \alpha_{e_1}$ from the terminal endpoint of $\nu_1'$ to the $g$--image of the initial endpoint.  Since $\widetilde \alpha_{e_1}$ is disjoint from $p^{-1}(w)$, it follows that $\nu_1''$ is disjoint from $p^{-1}(w)$.
	Now set
	\[ \widetilde \nu = \bigcup_{n \in \mathbb Z} g^n(\nu_1''), \] 
	which is a bi-infinite, $g$--invariant path.  Furthermore, since $w$ is a simple closed curve contained in an identity complementary region, $p^{-1}(w)$ is invariant by $g$ as well.  In particular,  $\widetilde \nu$ is also disjoint from $p^{-1}(w)$.  That is, $\widetilde \nu$ is contained in a single component of $\mathbb H^2 \ssm p^{-1}(w)$, and therefore, $g$ is contained in the stabilizer of this set.  The closure of this component is $\hull_{u_0}$ for some curve $u_0$ with $\Phi(u_0) = w$, and therefore $g$ fixes $u_0$, contradicting the fact that $G$ is purely pseudo-Anosov.  This contradiction shows that Case 1 cannot happen.
	
	\medskip
	
	\noindent {\bf Case 2}.  For every vertex $t$ of $\nu_1$ we have $\pi_t (v) = \emptyset$.
	
	\medskip
	
	Under these assumptions, we note that $v$ is disjoint from every complementary subsurface $Y_j$ that $p(\nu_1)$ intersects.  In particular, there must be some curve $\alpha_i \in \alpha$ that is disjoint from $p(\nu_1)$.  We can then build $\nu_1''$ and $\widetilde \nu$ as we did above, but in this case, the bi-infinite path $\widetilde \nu$ is disjoint from $p^{-1}(\alpha_i)$ instead of $p^{-1}(w)$.  Since $p^{-1}(\alpha_i)$ is invariant by $g$, we again find that  $\widetilde \nu$ is contained in a set $\hull_{u_0}$ where $u_0$ is a curve with $\Phi(u_0) = \alpha_i$.  As before, this implies  $g$ fixes $u_0$, which is another contradiction.  Therefore, Case 2 cannot happen either.  Since these two cases account for all possibilities, we see that the assumption that $\gamma$ had length greater than $2E+2$ was impossible. 
\end{proof}

The lemma above uniformly bounds the length of any subsegment of $p_0(T_{u,G}^\hull)$ that  is outside of the set $\mathcal E(v,B_2) \cup \mathcal V(v,B_2)$. Combining this with the fact that  $\mathcal E(v,B_2)$ and $\mathcal V(v,B_2)$ are finite collections of uniformly bounded diameter sets (Lemma \ref{L:bounded clusters}), we can produce a uniform bound of $\diam(p_0(T_{u,G}^\hull))$.

\begin{proof}[Proof of Lemma~\ref{L:hull subtree bound}]  Recall that $E$ and $V$ respectively denote the number of $G/G_0$--orbits of edges and vertices in $\sigma_0$.
	Let $M > 0$ be the constant from Lemma~\ref{L:bounded clusters} for $B = B_2$ and set $D_2 = (E+V)(2M+2E+3)$. 
	
	Recall that $\mathcal E(v,B_2) \cup \mathcal V(v,B_2) \subset \sigma_0$ is a union of at most $E+V$ sets of diameter at most $M$ by Lemma~\ref{L:bounded clusters}.  If $T_{u,G}^\hull$ does not intersect any of these sets, then Lemma~\ref{L:lengths outside} says $\diam(p_0(T_{u,G}^\hull)) \leq 2E+2\leq D_2$. Otherwise,  let $L \leq E+V$ be the number of these sets that non-trivially intersect $p_0(T_{u,G}^\hull)$, and let $X_1,\ldots, X_L$ be these sets (whose diameters are thus at most $M$).  According to Lemma~\ref{L:lengths outside}, the maximal length of an edge path in $p_0(T_{u,G}^\hull)$ outside the union of these sets is at most $2E+2$. Therefore the maximum distance from any point of $p_0(T_{u,G}^\hull)$ to $X_1 \cup \ldots \cup X_L$ is at most $2E+3$.  The collection $\{N_{M+2E+3}(X_j)\}_{j=1}^L$ is then a connected cover of $p_0(T_{u,G}^\hull)$, where each  set has diameter at most $2M+2E+3$. Thus, the diameter of $p_0(T_{u,G}^\hull)$ is at most $L(2M+2E+3) \leq (E+V)(2M+2E+3) =  D_2$, as required.
	\end{proof}

\subsection{Combining bounds}

The proof of Proposition~\ref{P:weak bounded diameter} is now straightforward.\\

\noindent
{\em {\bf Proposition~\ref{P:weak bounded diameter}.} \Pweakbound}

\bigskip

\begin{proof} Observe that $T_u \cap T_G$ is a union of $T^\hull_{u,G}$ and some set of parallel type subtrees, each of which is connected by an edge to $T^\hull_{u,G}$.  Therefore, by Lemmas~\ref{L:bound_on_parallel_type_diam} and \ref{L:hull subtree bound}, the diameter of $p_0(T_u \cap T_G)$ is at most $2D_1+D_2+2$.
\end{proof}

\bibliographystyle{alpha}
\bibliography{main}

\end{document}